\documentclass[a4paper]{amsart}
\pdfoutput=1
\usepackage[latin1]{inputenc}
\usepackage[l2tabu, orthodox]{nag}
\usepackage{amsfonts,enumerate,mathtools}

\usepackage{amssymb}
\usepackage{amsthm}
\usepackage{amsmath}
\usepackage{amsxtra}
\usepackage{array}
\usepackage{booktabs}
\usepackage{mathrsfs}
\usepackage{multirow}
\usepackage{tikz-cd}
\usepackage[all]{xy}

\usepackage[pagebackref,breaklinks]{hyperref}

\newcommand{\into}{\hookrightarrow}

\newcommand{\RR}{\mathbf{R}}
\newcommand{\CC}{\mathbf{C}}
\newcommand{\QQ}{\mathbf{Q}}
\newcommand{\QQbar}{\overline{\QQ}}
\newcommand{\Qlbar}{\overline{\QQ}_\ell}
\newcommand{\ZZ}{\mathbf{Z}}

\renewcommand{\AA}{\mathbf{A}}

\newcommand{\AF}{\AA_F}
\newcommand{\AEf}{\AA_{E, \mathrm{f}}}
\newcommand{\AFf}{\AA_{F, \mathrm{f}}}

\newcommand{\cO}{\mathcal{O}}
\newcommand{\cT}{\mathcal{T}}
\newcommand{\cS}{\mathcal{S}}

\newcommand{\frakH}{\mathfrak{H}}

\newcommand{\fd}{\mathfrak{d}}
\newcommand{\fp}{\mathfrak{p}}
\newcommand{\fq}{\mathfrak{q}}
\newcommand{\fN}{\mathfrak{N}}
\newcommand{\fm}{\mathfrak{m}}

\newcommand{\uk}{\underline{k}}
\newcommand{\ut}{\underline{t}}

\DeclareMathOperator{\GL}{GL}
\DeclareMathOperator{\SL}{SL}
\DeclareMathOperator{\Cl}{Cl}
\DeclareMathOperator{\Nm}{Nm}
\DeclareMathOperator{\tr}{tr}
\DeclareMathOperator{\Res}{Res}
\DeclareMathOperator{\Frob}{Frob}

\DeclareMathOperator{\Gal}{Gal}
\DeclareMathOperator{\GO}{GO}
\DeclareMathOperator{\GSO}{GSO}
\DeclareMathOperator{\Ind}{Ind}

\DeclareMathOperator{\nrd}{nrd}

\newcommand{\tbt}[4]{\begin{pmatrix}#1 & #2 \\ #3 & #4 \end{pmatrix}}
\newcommand{\stbt}[4]{
 \left(\begin{smallmatrix}#1 & #2 \\ #3 & #4\end{smallmatrix}\right)
}

\newcommand{\eps}{\varepsilon}
\newcommand{\fn}{\mathfrak{n}}
\newcommand{\Sym}{\operatorname{Sym}}

\newtheorem{definition}{Definition}[section]
\newtheorem{theorem}[definition]{Theorem}
\newtheorem{proposition}[definition]{Proposition}
\newtheorem{conjecture}[definition]{Conjecture}
\newtheorem{lemma}[definition]{Lemma}
\newtheorem{corollary}[definition]{Corollary}

\theoremstyle{remark}
\newtheorem{remark}[definition]{Remark}
\newtheorem*{notation}{Notation}

\numberwithin{table}{section}

\begin{document}

\title{Non-paritious Hilbert modular forms}

\author{Lassina Demb\'el\'e}
\address[Demb\'el\'e]{Max-Planck Institute for Mathematics, Vivatsgasse 7, Bonn 53111, Germany}
\email{lassina.dembele@gmail.com}
\thanks{At various stages of this project, LD was supported by EPSRC Grant EP/J002658/1 and a Visiting Scholar grant from the Max-Planck Institute for Mathematics}

\author{David Loeffler}
\address[Loeffler]{Mathematics Institute, Zeeman Building, University of Warwick, Coventry CV4 7AL, UK. ORC ID: 0000-0001-9069-1877.}
\email{d.a.loeffler@warwick.ac.uk}
\thanks{DL was supported by a Royal Society University Research Fellowship.}

\author{Ariel Pacetti}
\address[Pacetti]{Leverhulme Trust visiting Professor, University of Warwick, Coventry CV4 7AL, UK.}
\email{apacetti@famaf.unc.edu.ar}
\thanks{AP was supported by a Leverhulme Trust Visiting Professorship}


\begin{abstract}
 The arithmetic of Hilbert modular forms has been extensively studied under the assumption that the forms concerned are ``paritious'' -- all the components of the weight are congruent modulo 2. In contrast, non-paritious Hilbert modular forms have been relatively little studied, both from a theoretical and a computational standpoint. 
 
 In this article, we aim to redress the balance somewhat by studying the arithmetic of non-paritious Hilbert modular eigenforms. On the theoretical side, our starting point is a theorem of Patrikis, which associates \emph{projective} $\ell$-adic Galois representations to these forms. We show that a general conjecture of Buzzard and Gee actually predicts that a strengthening of Patrikis' result should hold, giving Galois representations into certain groups intermediate between $\GL_2$ and $\operatorname{PGL}_2$; and we verify that the predicted Galois representations do indeed exist. On the computational side, we give an algorithm to compute non-paritious Hilbert modular forms using definite quaternion algebras. To our knowledge, this is the first time such a general method has been presented. We end the article with an example.
\end{abstract}

\maketitle
\setcounter{tocdepth}{1}
\tableofcontents

\section*{Introduction}

 \subsection*{Background} Let $G$ be a reductive group over a number field $F$. One of the key themes of the Langlands programme is that ``sufficiently nice'' automorphic representations of $G$ should give rise to $\ell$-adic Galois representations, for any prime $\ell$. However, translating this idea into a formal statement is surprisingly difficult, and a precise formulation of such a conjecture has only recently been given by Buzzard and Gee in \cite{buzzardgee}.
 
 In \emph{op.cit.}, they define a class of automorphic representations $\Pi$ of $G$ which are ``$L$-algebraic''; and their conjecture predicts that if $\Pi$ is $L$-algebraic, then for every prime $\ell$ (and isomorphism $\CC \cong \Qlbar$), there should be a continuous representation of $\Gal(\overline{F} / F)$ with values in the Langlands $L$-group ${}^L G(\Qlbar)$, whose restrictions to the decomposition groups at good primes $v$ are determined by the corresponding local factors $\Pi_v$ of $\Pi$. (We shall recall the statement of this conjecture in more detail below.)
 
 One natural testing ground for this conjecture is provided by Hilbert modular forms. As noted in \emph{op.cit.}, if $F$ is a totally real number field, and $f$ is a Hilbert modular form for $\GL_2 / F$, then the automorphic representation $\Pi$ associated to $f$ is $L$-algebraic (after a suitable twist) if and only if the weight of $f$ is ``paritious'' (all of its components $k_\sigma$ are congruent modulo 2). It is well-known that paritious Hilbert eigenforms have associated 2-dimensional $\ell$-adic Galois representations, confirming the Buzzard--Gee conjecture in this case. 
 
 On the other hand, there are also eigenforms that are non-paritious. These do not have 2-dimensional Galois representations; however, Patrikis \cite{patrikis-sign} showed\footnote{Patrikis' result is actually considerably more general, applying to regular algebraic, essentially self-dual cuspidal automorphic representations of $\GL_n$ over totally real fields. However, we shall consider only the $n = 2$ case in the present paper.} one can associate 2-dimensional \emph{projective} $\ell$-adic Galois representations to such forms. This is wholly consistent with the Buzzard--Gee conjecture: the group $\operatorname{PGL}_2$ is the Langlands dual of $\SL_2$, and one checks that non-paritious eigenforms give rise to automorphic representations of $\GL_2$ which are not $L$-algebraic, but become $L$-algebraic when restricted to $\SL_2$. This has inspired us to begin a more general study of non-paritious Hilbert modular forms, both from a theoretical and a computational viewpoint; as far as we are aware, the problem of computing non-paritious forms explicitly has not been considered before.
 
 \subsection*{Goals of this article}
 
 The goals of the present article are the following. 
 \begin{enumerate}
  
  \item We introduce a hierarchy of conditions on the weight $(\uk, \ut)$ of a Hilbert modular automorphic representation $\Pi$ for $\GL_2 / F$, depending on a choice of a subfield $E \subseteq F$; we call such weights ``$E$-paritious''. (If $E = F$, this is the usual parity condition that all the $k_\sigma$ are  congruent modulo 2. If $E = \QQ$ it is no condition at all, i.e.~every $\Pi$ is $\QQ$-paritious). We define a subgroup $G^*$ of the restriction of scalars $G \coloneqq \Res_{F / E} \GL_2$, containing $\Res_{F / E} \SL_2$; and we show that if $\Pi$ is $E$-paritious, the restriction of $\Pi$ to $G^*(\AA_E)$ \textbf{is} $L$-algebraic after a suitable twist.
   
  \item We shall demonstrate that, as predicted by the Buzzard--Gee conjecture, we may associate $\ell$-adic representations of $\operatorname{Gal}(\overline{E} / E)$ to $E$-paritious automorphic representations of $\GL_2 / F$, taking values in the Langlands $L$-group of the group $G^*$ defined in (1). Since our group $G^*$ always strictly contains $\Res_{F / E}(\SL_2)$, whose Langlands dual is $\Res_{F / E}(\operatorname{PGL}_2)$, this result refines Patrikis' construction of projective Galois representations.
  
  \item We describe algorithms for computing non-paritious Hilbert modular forms, via the Jacquet--Langlands correspondence between $\GL_2$ and totally definite quaternion algebras.
  
  \item We give an explicit example of non-paritious Hilbert modular forms computed using these algorithms, and describe the conjugacy classes of Frobenius elements in their associated Galois representations.
 \end{enumerate}

 The article is organized as follows: in Section $1$ we state
 Buzzard-Gee conjecture, and make a small detour through the concepts
 involved. Section $2$ is about Hilbert modular forms: we recall their
 automorphic definition, and we prove that if a non-paritious Hilbert
 modular form is $E$-paritious (see
 Definition~\ref{definition:E-paritious}) then we can restrict it to
 an automorphic form of $G^* = G \times_{(\Res_{F / E} \GL_1)} \GL_1$
 (as predicted by Buzzard-Gee). Section $3$ contains the main theorem
 (Theorem~\ref{thm:representation}), namely that non-paritious Hilbert
 modular forms, do have Galois representations attached to them, as
 predicted. Section $4$ relates our construction with Patrikis'
 one. In Section $5$ we focuss on real quadratic fields, where some
 exceptional isomorphism allows the Galois representation to land in
 $\GO_4$. In Section $6$ we show how to use quaternion groups to
 compute Hilbert modular forms (paritious and non-paritious ones). In
 particular, in Theorem~\ref{thm:fibers} and
 Corollary~\ref{corollay:algorithm} we prove how from automorphic
 forms for the quaternion group $H$ we can construct forms in
 $H^*$. This is the key result for computational purposes. In the same
 section we explain how to compute the Hecke action on such forms. We
 end the article with one illustrative example.  The code used is available at \url{https://warwick.ac.uk/fac/sci/maths/people/staff/david_loeffler/research/nonparitious/}

 \subsection*{Notation} Throughout the article, we use the following notations:
 \begin{itemize}
 \item $F$ denotes a number field. (In \S \ref{sect:lgroups} $F$ can be arbitrary, but from \S \ref{sect:HMF} onwards we shall assume $F$ to be totally real.)
 
 \item $\cO_F$ denotes the ring of integers of $F$, $\cO_F^\times$ the unit group, and $\cO_F^{\times +}$ the subgroup of totally positive units.
 
 \item $\AA_F$ is the ad\`ele ring of $F$. 
    
 \item $\Cl^+(F)$ denotes the narrow class group of $F$.

 \item $\Gamma_F$ denotes the Galois group $\Gal(\overline{F}/F)$.
 
 \item $E$ will denote a subfield of $F$, and the notations $\cO_E$, $\Gamma_E$ etc have the same meanings as for $F$.
   
 \end{itemize}

 \subsection*{Acknowledgements} It is a pleasure to thank the two authors of the conjecture we are studying: firstly, Kevin Buzzard for several helpful remarks, and in particular for pointing us towards the work of Blasius--Rogawski which is the key input to constructing the required Galois representations; and secondly, Toby Gee, for making us aware of the related work of Patrikis. We are also grateful to Stefan Patrikis for his comments on an earlier version of this paper.

\section{L-groups}
 \label{sect:lgroups}
 
 In this section we'll recall from \cite{buzzardgee} the necessary notions to formulate their conjecture relating automorphic representations and Galois representations; and we will check the compatibility of their conjecture with restriction of scalars.
 
 \subsection{Global definitions}
 
  Let $G$ be a connected reductive group over a number field $F$. The \emph{Langlands dual} $\hat G$ is the connected reductive group $\hat G$ over $\QQbar$ whose root datum is dual to that of $G$. The Galois group $\Gamma_F = \Gal(\overline{F} / F)$ acts naturally on $\hat G$, and the \emph{Langlands $L$-group} ${}^L G$ is the pro-algebraic group over $\QQbar$ defined as the semidirect product $\hat G \rtimes \Gamma_F$. See \cite[\S 2.1]{buzzardgee} for details. If $G$ is split over $F$ (or is an inner form of a split group) the action of $\Gamma_F$ on $\hat G$ is trivial, so ${}^L G$ is a direct product.
  
  We shall be interested in continuous homomorphisms $\rho: \Gamma_F \to {}^L G(M)$, for various fields $M$, satisfying the following condition: \emph{the composite of $\rho$ with the projection ${}^L G(M) \to \Gamma_F$ is the identity map on $\Gamma_F$.} Such a morphism is called an \emph{admissible homomorphism}, or sometimes \emph{$L$-homomorphism}. More generally, if $\Gamma' \subseteq \Gamma_F$ is a subgroup, we define a homomorphism $\Gamma' \to {}^L G(M)$ to be admissible if its projection to $\Gamma_F$ is the inclusion map $\Gamma' \into \Gamma_F$.
  
  \begin{notation}
   If $H_1$ and $H_2$ are two reductive groups over $F$, then the Langlands $L$-group ${}^L(H_1 \times H_2)$ is the fibre product ${}^L H_1 \times_{\Gamma_F} {}^LH_2$; for $r_1: \Gamma_F \to {}^L H_1$ and $r_2: \Gamma_F \to {}^L H_2$ admissible homomorphisms, we write $r_1 \times r_2: \Gamma_F \to {}^L(H_1 \times H_2)$ for their product.
  \end{notation}
 
 \subsection{Local theory}
 
  If $v$ is a finite place of $F$ at which $G$ is unramified (i.e.\ $G$ is quasi-split over $F_v$ and becomes split over an unramified extension of $F_v$), then there is a parametrisation of unramified representations of $G(F_v)$ in terms of \emph{Langlands--Satake parameters}. We choose an embedding $\overline{F} \into \overline{F_v}$, so we can identify $\Gamma_{F_v}$ with a subgroup of $\Gamma_F$. Then a Langlands--Satake parameter is a $\hat{G}(\CC)$-conjugacy class of admissible homomorphisms
  \[ s_v: W_{F_v} \to {}^L G(\CC) \]
  whose projection to $\hat G(\CC) \rtimes \Gal(F_v^{\mathrm{nr}}/F_v)$ factors through $W_{F_v} / I_{F_v}$, where $I_{F_v}$ is the inertia group, and satisfies a certain semisimplicity condition. (Note that this projection is well-defined, since the action of the inertia group $I_v$ on $\hat G(\CC)$ is trivial by assumption.) 
  
  If $\Gamma_{F_v}$ acts trivially on $\hat G$ -- equivalently, if $G$ is split over $F_v$ -- then $s_v$ is entirely determined by the conjugacy class of the projection to $\hat G(\CC)$ of $s_v(\Frob_v)$. This semisimple conjugacy class in $\hat G(\CC)$ is referred to simply as a \emph{Satake parameter}.
  
  As explained in \cite[\S 2.2]{buzzardgee}, there is a bijection between isomorphism classes of irreducible unramified representations of $G(F_v)$, and Langlands--Satake parameters.
 
 \subsection{The Buzzard--Gee conjecture}
 
  Let $\Pi = \bigotimes' \Pi_v$ be an automorphic representation of $G(\AA_F)$. Then the local factor $\Pi_v$ is unramified for almost all $v$, so we have a collection of Satake parameters $(s_v)_{v \notin \Sigma}$, where $\Sigma$ is a finite set.
  
  On the other hand, we also have a \emph{Harish--Chandra parameter} for each infinite place $\sigma$ of $F$, which is a Weyl group orbit\footnote{If $\sigma$ is a complex place then there is a small subtlety in that $\lambda_\sigma$ actually depends not only on the place $\sigma$ but also on a choice of isomorphism $F_\sigma \cong \CC$; but replacing this isomorphism with its conjugate changes $\lambda_\sigma$ by an element of $X_\bullet(\hat T)$, so the notion of $L$-algebraicity is well-defined. However, in this paper we shall mostly restrict to the case of totally real $F$ where this subtlety does not arise.} $\lambda_\sigma \in X_\bullet(\hat T) \otimes \CC$, where $\hat T$ is a maximal torus in $\hat G$. 
  
  \begin{definition}
   We say $\Pi$ is \emph{$L$-algebraic} if $\lambda_\sigma \in X_\bullet(\hat T)$ for every infinite place $\sigma$.
  \end{definition}
  
  \begin{conjecture}[{\cite[Conjectures 3.1.1 \& 3.2.1]{buzzardgee}}]
   Suppose $\Pi$ is an $L$-algebraic automorphic representation of $G(\AA_F)$. Then there is a finite extension $E / \QQ$ such that the Satake parameters $r(\Pi_v)$ are all defined over $E$; and for any prime $\ell$ and choice of embedding $\iota: E \into \Qlbar$, there exists an admissible homomorphism
   \[ r_\Pi: \Gamma_F \to {}^L G(\Qlbar) \]
   such that the restriction of $r_\Pi$ to $\Gamma_{F_v}$ is conjugate to $\iota(s_v)$ for every prime $v \notin \Sigma$ such that $v \nmid \ell$.
  \end{conjecture}
 
 \subsection{Weil restriction}
  
  We now check a compatibility property of the above conjecture. Let $E \subseteq F$ be number fields. Let $H$ be a reductive group over $F$, and let $G$ be the Weil restriction $\Res_{F / E} H$, which is a reductive group over $E$. Then $G(\AA_E)$ is canonically isomorphic to $H(\AA_F)$, and this isomorphism sends $G(E)$ to $H(F)$; so automorphic representations of $H(\AA_F)$ and of $G(\AA_E)$ are the same objects. However, the Buzzard--Gee conjecture for $H$ over $F$, and for $G$ over $E$, are apparently very different statements. In this section we shall check that the two statements are in fact equivalent.
  
  \begin{proposition}
  \label{prop:induceLhom}
  Let $E \subseteq F$ be number fields. Let $H$ be a reductive group over $F$, and let $G$ be the Weil restriction $\Res_{F / E} H$, which is a reductive group over $E$. Then:
  \begin{itemize}
   \item The dual group $\hat G$ is a product of $[F : E]$ copies of $\hat H$ indexed by the cosets $\Gamma_E / \Gamma_F$; in particular the subgroup $\Gamma_F$ preserves the first factor.
   \item The $L$-group ${}^LG$ is isomorphic to the semidirect product $\hat{G}\rtimes \Gamma_E$, with the natural action of $\Gamma_E$ on $\hat{G}$. 
   \item If $r: \Gamma_F \to {}^L H(\Qlbar)$ is an admissible homomorphism, there is an admissible homomorphism
   \[ \tilde r = \Ind_{F/E}(r): \Gamma_E \to {}^L G(\Qlbar)\]
   (uniquely determined up to conjugacy) such that the projection of $\tilde r |_{\Gamma_F}$ to the first factor of $\hat G$ is $r$.
  \end{itemize}
 \end{proposition}
 
 \begin{remark}
  This proposition takes a particularly simple form if $H$ is split over $F$ (or is an inner form of a split group). In this case the action of $\Gamma_F$ on $\hat H$ is trivial, so ${}^L H$ is a direct product; and an admissible homomorphism $\Gamma_F \to {}^L H (\Qlbar)$ is simply a homomorphism $\Gamma_F \to \hat H(\Qlbar)$. Meanwhile, $\hat G \cong \prod_{x \in \Gamma_E / \Gamma_F} \hat H$, with $\Gamma_E$ acting by permuting the factors via its left action on $\Gamma_E / \Gamma_F$. 
 
  In this situation, if $r$ is an $L$-homomorphism $\Gamma_F \to {}^L H(\Qlbar)$, and $\rho: \hat H \to \GL_m$ is a representation of $\hat H$, then there is a natural representation $\tilde \rho: {}^L G \to \GL_{[F : E] m}$ whose restriction to the identity component $\hat G$ is given by $\rho \times \dots \times \rho$; and the composite $\tilde \rho \circ \tilde r$ is the induced representation $\Ind_{\Gamma_F}^{\Gamma_E} (\rho \circ r)$ in the usual sense. This justifies the notation ``$\Ind_{F / E}(r)$'' for this homomorphism $\tilde r$.
 \end{remark}

 \begin{proof}[Proof of Proposition \ref{prop:induceLhom}]
  The first two statements of the proposition are standard. We give an outline of the construction of the homomorphism $\tilde r$. 
  
  It is convenient to work in a slightly more general setting: let $V$ be an arbitrary group, and $\rho: V \to H$ a homomorphism. Suppose $U \ge V$ is an overgroup with $[U : V] = d < \infty$. 
   
  Let $G$ be the group $H^{U / V} \rtimes U$. Explicitly, an element of $G$ is a pair $(f, u)$ where $f$ is a function $U/V \to H$ and $u \in U$, and the multiplication is given by $(f, u) (f', u) = (x \mapsto f(x) f'(u^{-1} x), u u')$.
  
  We define a map $\tilde \rho: U \to G, u \mapsto (f_u, u)$, where $f_u: U/V \to H$ is defined as follows. Choose a set of coset representatives $U = \bigsqcup_{i = 1}^d u_i V$. We define $f_u(u_i) = \rho(u_i^{-1} u u_k)$, where $k \in \{1, \dots, d\}$ is the unique index such that $u_i^{-1} u u_k \in V$. Then a routine but tedious check shows that $\tilde \rho$ is a group homomorphism.
 \end{proof}
 
 We now consider automorphic representations of $G$ and $H$. Let $\Pi$ be an automorphic representation of $H(\AA_F)$, and let $\tilde \Pi$ denote the same space regarded as a representation of $G(\AA_E)$.
 
 \begin{proposition}
  We have the following compatibilities:
  \begin{enumerate}[(i)]
   
   \item $\Pi$ is $L$-algebraic as a representation of $G(\AA_E)$ if and only if $\tilde \Pi$ is $L$-algebraic as a representation of $H(\AA_F)$ \cite[\S 3.1]{buzzardgee}.
   
   \item If $w$ is a finite place of $E$ such that $F_v / E_w$ is unramified for every $v \mid w$, then $\tilde \Pi_w = \bigotimes_{v \mid w} \Pi_v$ is unramified as a representation of $G(E_w)$ if and only if each $\Pi_v$ is unramified as a representation of $H(F_v)$; and in this setting, the Langlands--Satake parameter $\tilde s_w$ of $\tilde \Pi_w$ is defined over a subfield $E$ if and only if the same is true of each of the $s_v$. 
   
   \item Let $r: \Gamma_F \to {}^L H(\overline\QQ_\ell)$ be an admissible homomorphism, and let $\tilde r: \Gamma_E \to {}^L G(\overline\QQ_\ell)$ be the induction of $r$ described in Proposition \ref{prop:induceLhom}. Then the restriction of $\tilde r$ to $W_{E_w}$ is $\hat G$-conjugate to $\iota( \tilde s_w )$ if and only if the restriction of $r$ to $W_{F_v}$ is $\hat H$-conjugate to $\iota(s_v)$ for all $v \mid w$.
  \end{enumerate}
 \end{proposition}

 \begin{proof}
  Statements (i) and (ii) are proved in \cite{buzzardgee}, in Section 3.1 and Section 3.2 respectively. So it remains to prove (iii), for which we need to make precise the relation between the Langlands--Satake parameters of $\tilde\Pi_w$ and $\Pi_v$.
  
  Let $H_v$ denote the base extension of $H$ to $F_v$, and similarly for $G_w$. Then we have $G_w = \prod_{v \mid w} \Res_{F_v / E_w} H_v$ as algebraic groups over $E_w$. For each $v$, we have a Langlands--Satake parameter $s_v: W_{F_v} \to {}^L H_v(\CC) = \hat H(\CC) \rtimes \Gamma_{F_v}$ attached to $\Pi_v$. Applying exactly the same induction process as before, we obtain an admissible homomorphism 
  \[ \tilde s_v = \Ind_{F_v / E_w}(s_v) : W_{E_w} \to \hat H(\CC)^{\Gamma_{F_v} / \Gamma_{E_w}} \rtimes \Gamma_{E_w}. \]
  From the definition of the Langlands--Satake parameter, one sees that $\tilde s_v$ is exactly the Langlands--Satake parameter of $\Pi_v$ considered as a representation of the $E_w$-points of the algebraic group $\Res_{F_v / E_w} H_v$ over $E_w$.
  
  There is a bijection between the orbits for the action of the Frobenius $\sigma_w$ on the factors of $\hat G(\CC)$, and the primes $v \mid w$; so taking the fibre product (over $\Gamma_{E_w}$) of the representations $\tilde s_v$ defines an admissible homomorphism $\tilde s_w: W_{E_w} \to {}^L G(\CC)$. Since the Langlands--Satake parameter of a representation $\Pi \otimes \Pi'$ of a product group $U \times U'$ is the fibre product of the parameters of the factors, we see that $\tilde s_w$ is exactly the Langlands--Satake parameter of $\tilde \Pi_w$. On the other hand, since $\tilde s_w$ is obtained from $(s_v)_{v \mid w}$ by induction, it is clear that $\iota(\tilde s_w)$ is the restriction to $W_{E_w}$ of a global homomorphism $\tilde r = \Ind_{F / E}(r)$ if and only if $\iota(s_v)$ is the restriction of $r$ to $W_{F_v}$ for all $v \mid w$.
 \end{proof}

 \begin{corollary}
  \label{cor:resscalars}
  The Buzzard--Gee conjecture is true for an automorphic representation $\Pi$ of $H(\AA_F)$ if, and only if, it is true for the same representation regarded as a representation of $G(\AA_E)$.\qed
 \end{corollary}

\section{Hilbert modular forms}
\label{sect:HMF}

\subsection{Weights}

Let $F$ be a totally real field, and let $\Sigma_F$ be the set of infinite places of $F$. By a \emph{weight} for $F$, we mean a collection $\uk = (k_\sigma)_{\sigma \in \Sigma_F}$ of integers indexed by $\Sigma_F$.

\begin{notation}
 For $x \in F^\times$ and $\uk$ a weight, we write $x^{\uk}$ for $\prod_\sigma \sigma(x)^{k_\sigma} \in \RR^\times$.
\end{notation}

Thus weights are just the same thing as characters of the torus $\Res_{F / \QQ} \mathbf{G}_m$.

\begin{definition}
 We say $k$ is \emph{paritious} if the parity of $k_\sigma$ is independent of $\sigma$.
\end{definition}

We also consider a slightly more general notion. For $E \subseteq F$ a subfield and $\uk$ a weight of $F$, we define $\uk_E$ to be the weight for $E$ defined by $(k_E)_\tau = \sum_{\sigma \mid \tau} k_\sigma$ (equivalently, the restriction of $\uk$ to $\Res_{E / \QQ} \mathbf{G}_m \subset \Res_{F / \QQ} \mathbf{G}_m$).

\begin{definition}
 We shall say $\uk$ is \emph{$E$-paritious} if $\uk_E$ is paritious as a weight for $E$.
\label{definition:E-paritious}
\end{definition}

Thus being $E$-paritious is no condition at all if $E = \QQ$, and becomes more restrictive as $E$ gets larger, with the opposite extreme $E = F$ being the previous definition.

 \subsection{Adelic Hilbert modular forms}
  \label{sect:adelicHMF}
  Let $\frakH_F$ be the set of elements of $F \otimes \CC$ of totally positive imaginary part, with its natural left action of $\GL_2^+(F \otimes \RR)$. Let $\uk = (k_\sigma)_{\sigma \in \Sigma_F}$ be a collection of integers, and $\ut = (t_\sigma)_{\sigma \in \Sigma_F}$ a collection of real numbers. We can define the weight $(\uk, \ut)$ right action of $\GL_2^+(F \otimes \RR)$ on functions $\frakH_F \to \CC$ by
  \[ 
   (f \mid_{k, t} \gamma )(\tau) =  \det(\gamma)^{\uk + \ut - 1} (c\tau + d)^{-\uk} f(\gamma \cdot \tau).
  \]
  
  \begin{notation}
   We say the pair $(\uk, \ut)$ is \emph{reasonable} if the quantity $k_\sigma + 2 t_\sigma$ is independent of $\sigma$, which is equivalent to requiring that $\stbt x00x$ acts trivially for all $x \in \cO_F^{\times +}$ (or just for all $x$ in a finite-index subgroup). We denote the common value of $k_\sigma + 2t_\sigma$ by $R$.
  \end{notation}
  
  We define a \emph{Hilbert modular form of weight $(\uk, \ut)$} to be a function 
  \[ f: \GL_2(\AFf) \times \frakH_F \to \CC\]
  such that 
  \begin{itemize}
   \item $f(g, -)$ is holomorphic on $\frakH_F$ for all $g \in \GL_2(\AFf)$,
   \item $f(\gamma g, -) = f(g, -) \mid_{\uk, \ut} \gamma^{-1}$ for all $\gamma \in \GL_2^+(F)$,
   \item there exists an open compact subgroup $U$ of $\GL_2(\AFf)$ such that $f(gu, \tau) = f(g, \tau)$ for all $u \in U$ and $(g, \tau) \in \GL_2(\AFf) \times \frakH_F$.
  \end{itemize}
  (If $F = \QQ$ we need an additional condition of holomorphy at the cusps, which is otherwise automatic by the K\"ocher principle.) We write $M_{\uk, \ut}$ for the space of such functions, and $S_{\uk, \ut}$ for the subspace of cusp forms. Both spaces are clearly zero unless $(\uk, \ut)$ is reasonable. From now on $(\uk,\ut)$ is implicitly assumed reasonable.
  
  \begin{remark}
   We have chosen to formulate the definition in terms of $\GL_2(\AFf) \times \frakH_F$ since it makes the link to the classical theory slightly more direct. The alternative, more analytic, approach is to work with functions on the quotient $\GL_2(F) \backslash \GL_2(\AF)$. Concretely, if $f$ is a Hilbert modular form in the above sense, then the function $\tilde f$ on $\GL_2(\AF)$ given by $\tilde f(g_\mathrm{fin}, g_\infty) = \left( f(g_\mathrm{fin}, -) \mid_{\uk, \ut} g_\infty\right)(1 \otimes i)$ is left $\GL_2(F)$-invariant, and for each $\sigma \in \Sigma_F$, it transforms by $e^{i k_\sigma \theta}$ under right translation by $\tbt{\cos \theta}{\sin \theta}{-\sin \theta}{\cos \theta} \in \mathrm{SO}_2(F_\sigma)$. Conversely we can recover $f$ from $\tilde f$ via $f(g, x + iy) = y^{-(\uk + \ut - 1)} \tilde f\left(g, \stbt y x 0 1\right)$.
  \end{remark}
  
  The following properties of $M_{\uk, \ut}$ and $S_{\uk, \ut}$ are well-known:
  
  \begin{itemize}
   
   \item The spaces $M_{\uk, \ut}$ and $S_{\uk, \ut}$ are admissible smooth representations of the group $\GL_2(\AFf)$, via the right-translation action. 
   
   \item If $\ut' = \ut + h \cdot \underline{1}$ for some $h \in \RR$, where $\underline{1}$ is the weight all of whose components are 1, then the map $f \mapsto f'$, $f'(g, \tau) = \|\det g\|^h f(g, \tau)$, defines a bijection between $M_{\uk, \ut}$ and $M_{\uk, \ut'}$, and an isomorphism of $\GL_2(\AFf)$-representations.
   \[ 
    M_{\uk, \ut'} = M_{\uk, \ut} \otimes \|\det\|^h.
   \]
   (Here $\|x \|$ is the ad\`ele norm map, sending a uniformiser at a prime $\fq$ of $F$ to the reciprocal of the size of its residue field.)
  
   \item For any $f \in M_{\uk, \ut}$ there is a finite-index subgroup of $\AFf^\times$, containing $F^{\times +}$, such that for $x$ in this subgroup, $\stbt x 0 0 x \in Z(\GL_2(\AFf))$ acts on $f$ by $\|x\|^{R - 2}$ where $R$ is the common value of $k_\sigma + 2t_\sigma$.

   \item If the $t_\sigma$ are all in $\ZZ$, then $M_{\uk, \ut}$ and $S_{\uk, \ut}$ are the base-extensions to $\CC$ of $\GL_2(\AFf)$-representations defined over $\tilde F$, the Galois closure\footnote{Actually a somewhat smaller space suffices: one can take here the fixed field of the largest subgroup of $\operatorname{Gal}(\tilde F / \QQ)$ whose permutation action on $\Sigma_F$ stabilises the weight $\uk$.} of $F$ in $\CC$ (see e.g.\ \cite{shimura78}).
  \end{itemize}
 \subsection{Hecke theory and Satake parameters}
   
  Let $\Pi$ be an irreducible $\GL_2(\AFf)$-subrepresentation of $S_{\uk, \ut}$. Then we can write $\Pi = \sideset{}{'}{\bigotimes_v} \Pi_v$, where the product runs over finite primes of $F$, and each $\Pi_v$ is an irreducible smooth representation of $\GL_2(F_v)$. All but finitely many of the $\Pi_v$ will be unramified, so we have a collection of Satake parameters $s_v$.
  
  These $s_v$ can be described in terms of the action of Hecke operators. Let $\cT(v)$ denote the double coset of $\stbt 1 0 0 {\varpi_v}$, where $\varpi_v \in \AFf$ is a uniformiser at $v$; and let $\cS(v)$ denote the double coset of $\stbt {\varpi_v} 0 0 {\varpi_v}$. If $\tau_v$ and $\sigma_v$ denote the eigenvalues of these operators acting on the $\GL_2(\cO_{F, v})$-invariants of $\Pi$, then one has the following formula:
  
  \begin{proposition}
   The Satake parameter $s_v$ is the semisimple conjugacy class such that
   \[ \tr s_v = \Nm(v)^{-1/2} \tau_v \quad\text{and}\quad \det s_v = \sigma_v. \]
  \end{proposition}
    
  We give a more explicit description of the $s_v$ if the prime $v$ is narrowly principal, generated by a totally positive element $\varpi$; compare \cite[\S 3.3]{buzzardgee} for $F = \QQ$. Let $f$ be the new vector of $\Pi$. Then the restriction of $f$ to $\frakH_F$ has a Fourier expansion
  \[ f(\tau) = \sum_{\substack{\alpha \in \mathfrak{d}_F^{-1} \\ \alpha \gg 0}} c(\alpha) \exp\left( 2 \pi i \tr(\alpha \tau)\right).
  \]
  There is a constant $t(\varpi)$, the ``naive Hecke eigenvalue'', such that $c(\varpi \alpha) = t(\varpi) c(\alpha)$ if $(\varpi, \alpha \mathfrak{d}_F) = 1$. This is related to the ``normalised Hecke eigenvalue'' $\tau_v$ above by
  \[ \tau_v = \varpi^{\underline{2}-\uk-\ut} t(\varpi).\]
  Meanwhile, the quantity $\sigma_v$ is simply $\Nm(v)^{2-R} \chi(\varpi)$, where $\chi$ is the finite-order character by which the diamond operators act on $F$.

  It is shown in \S 3.2 of \cite{buzzardgee} that $\Pi$ is $L$-algebraic if and only if $t_\sigma \in \tfrac12 + \ZZ$, for all $\sigma \in \Sigma_F$.  Notice that, for a given $\uk$, we can find $\ut$ such that $(\uk, \ut)$ is reasonable and $t_\sigma \in \tfrac{1}{2} + \ZZ\ \forall \sigma$ if and only if $\uk$ is paritious. Thus the automorphic representations of $G$ arising from non-paritious Hilbert modular forms \emph{cannot} be twisted to become $L$-algebraic.
  
  It follows from Shimura's algebraicity theorem quoted above that if all $t_\sigma$ are in $\tfrac12 + \ZZ$ then the Satake parameters $s_v$ are all defined over a finite extension of $\QQ$ (for all good primes $v$, not only those trivial in the narrow class group).

  \begin{remark}
   \label{rem:LaritvsLalg}
   Buzzard and Gee define $\Pi$ to be \emph{$L$-arithmetic} if all the $s_v$ lie in a common finite extension. So Shimura's algebraicity theorem shows that if $\Pi$ is $L$-algebraic, then it is $L$-arithmetic. If $F = \QQ$, the converse holds: $L$-arithmetic implies $L$-algebraic, as shown in \cite{buzzardgee}. The same holds over general fields $F$, as we will see in the next section.
  \end{remark}

 \subsection{The group $G^*$}
 
  Now let $E$ be a subfield of $F$, as before, and set $G = \Res_{F/E} \GL_2$. We are interested in subgroups of $G$ defined by a condition on the determinant, as follows. The group $\GL_1$ is a subgroup of $\Res_{F / E} \GL_1$ in the obvious way. We define a group $G^*$ over $E$ by
  \[ G^* = G \times_{(\Res_{F / E} \GL_1)} \GL_1.\]
  Thus $G^*(E) = \{ g \in \GL_2(F): \det(g) \in E^*\}$.
  
  \begin{proposition}[{cf.\ \cite[p399]{brylinskilabesse84}}]
   The $L$-group of $G^*$ is the quotient of ${}^L G$ by a subgroup of $Z(\hat G)$. More specifically, if $K$ is the kernel of the ``norm'' map $Z(\hat G) = \prod_{\Gamma_E / \Gamma_F}\GL_1 \to \GL_1$, then $K$ is normal in ${}^L G$, and we have
   \[ \widehat {G^*} = \hat G / K, \qquad {}^L G^* = {}^L G / K. \]
  \end{proposition}
  
  \begin{remark}
   The group $\hat G = (\GL_2)^{\Gamma_E / \Gamma_F}$ has a $2^d$-dimensional representation, where $d = [F : E]$, given by the tensor product of the standard 2-dimensional representations of the $\GL_2$ factors. This representation factors through $\hat G^*$, and since it is invariant under permutation of the factors, it extends to a representation of ${}^L G^*$. We call this the \emph{Asai representation}, as the corresponding $L$-series first appeared in the work of Asai \cite{asai77}; see also Yoshida \cite{yoshida94}. However, it is important to note that many other interesting algebraic representations of ${}^L G$ factor through ${}^L G^*$, such as the induction from ${}^L H$ of the 3-dimensional adjoint representation of ${}^L H$, where $H = \GL_2 / F$.
  \end{remark}
  
  The reason for introducing $G^*$ is that it, so to speak, ``makes more representations algebraic''. There is a natural quotient map $X_\bullet(\hat T)$ to $X_\bullet(\hat T^*)$, where $\hat T$ is the standard maximal torus of $\hat G$. If $\lambda \in X_\bullet(\hat T)_\CC$, and $\lambda^*$ is its image in $X_\bullet(\hat T^*)_\CC$, then \textbf{it can occur that $\lambda^*$ is integral even if $\lambda$ is not}. In fact, we have the following result:

 \begin{proposition}
  Let $\Pi$ be the automorphic representation of $G(\AA_E) = \GL_2(\AA_F)$ given by a Hilbert modular form over $F$ of weight $(\underline{k}, \underline{t})$; and for $\tau$ a real place of $E$, let $\lambda_\tau$ be the Harish--Chandra parameter of $\Pi_\tau$. 
  
  Then the projection $\lambda_\tau^*$ lies in the integral cocharater lattice $X_\bullet(\hat T^*)$ if, and only if, we have $\sum_{\sigma \mid \tau} \left(t_\sigma - \tfrac{1}{2}\right) \in \ZZ$.
 \end{proposition}
 
 \begin{proof}
  Using the basis of the Cartan subalgebra of $\mathfrak{gl}_2(\CC)$ described in \cite[\S 3.3]{buzzardgee}, we can identify $X_\bullet(\hat T)$ with the abelian group 
  \[ 
   \left\{ (m_\sigma, n_\sigma)_{\sigma \mid \tau} : m_\sigma, n_\sigma \in \ZZ, m_\sigma = n_\sigma \bmod 2 \right\}\relax,
  \] 
  and in terms of this basis we have
  \[ \lambda_\tau = \Big( \pm (k_\sigma-1), k_\sigma + 2 t_\sigma - 2\Big)_{\sigma \mid \tau}.\]
  
  One has a similar description of $X_\bullet(\hat T^*)$; it is given by pairs $( (m_\sigma)_{\sigma \mid \tau}, n)$, with $m_\sigma, n \in \ZZ$ such that $n = \sum m_\sigma \bmod 2$. The quotient map is given by $(m_\sigma, n_\sigma)_{\sigma \mid \tau} \mapsto \left( (m_\sigma)_{\sigma \mid \tau}, \sum n_\sigma\right)$. So one computes that $\lambda_\tau^* \in X_\bullet(\hat T^*)$ if and only if $\sum_{\sigma} (t_\sigma - \tfrac12) \in \ZZ$, as required.
 \end{proof}

 \begin{proposition}
  If $\underline{k}$ is $E$-paritious, then we may choose the $t_\sigma$ such that $(\underline{k}, \underline{t})$ is reasonable and $\lambda_\tau^*$ is $L$-algebraic for all real places $\tau$ of $E$. Conversely, if $\underline{k}$ is not $E$-paritious then no such $\underline{t}$ exists.
 \end{proposition}
 \begin{proof}
   Since $(\uk,\ut)$ is reasonable, the quantity $k_\sigma + 2t_\sigma = R$
   is independent of $\sigma$. Then
   $\sum_{\sigma \mid \tau} (t_\sigma -\frac{1}{2}) =
   \frac{[F:E](R-1)-\sum_{\sigma \mid \tau}k_\sigma}{2}$.  We can
   chose $R$ so that this number is an integer if and only if the
   parity of $\sum_{\sigma \mid \tau} k_\sigma$ is independent of
   $\tau$.
 \end{proof}
 
 \subsection{Restriction of automorphic representations for $G$}
 
  Let $\Pi$ be an irreducible $\GL_2(\AFf)$-subrepresentation of $S_{\uk, \ut}$. Then we may consider the restriction of $\Pi$ to the subgroup $G^*(\AEf)$. This will usually not be irreducible. We denote by $\Psi$ the set of irreducible constituents of $\Pi$ as a $G^*(\AEf)$-representation; this is (the finite part of) a global $L$-packet for $G^*$.
  
  If $\Pi$ is not of CM type (which we shall assume from now on), then all representations $\Pi^* \in \Psi$ are the finite parts of automorphic representations of $G^*$, and they all have the same multiplicity in the spectrum of $G^*$ \cite[\S 3.2]{brylinskilabesse84}. Moreover, any two representations $\Pi_1^*, \Pi_2^* \in \Psi$ have the same Satake parameter at any prime where they are both unramified, and the same Harish--Chandra parameter at $\infty$; these parameters are simply the images of the Satake and Harish--Chandra parameters of $\Pi$ under the quotient map ${}^L G(\CC) \to {}^L G^*(\CC)$.
    
  In particular, the Buzzard--Gee conjecture is true for one $\Pi^* \in \Psi$ if and only if it holds for all of them, with the same representation $r_{\Pi^*, \iota}$. (That is, the Buzzard--Gee conjecture is really an assertion about automorphic $L$-packets, not about individual automorphic representations.)
  
 \section{Galois representations}
 
  \subsection{Setup}
  
  The following theorem, which establishes the Buzzard--Gee conjecture for automorphic representations of $\GL_2$ arising from paritious Hilbert modular forms, is well known:
 
 \begin{theorem}[Blasius--Rogawski]
  \label{thm:BRT}
  Let $\Pi$ be an irreducible subrepresentation of $S_{\uk, \ut}$, where $k_\sigma \ge 2$ and $t_\sigma \in \tfrac12 + \ZZ$ for all $\sigma$. Let $\ell$ be prime and let $\iota$ be an isomorphism $\CC \to \Qlbar$. Then there exists a continuous Galois representation
  \[ r_{\Pi, \iota}: \Gamma_F \to \GL_2(\Qlbar) \]
  such that for all primes $v \nmid \ell$ at which the local factor $\Pi_v$ is unramified, the representation $ r_{\Pi, \iota}$ is also unramified, and the conjugacy class of $r_{\Pi, \iota}(\Frob_v)$ is $\iota(s_v)$.
 \end{theorem}
  
 (For concreteness we take $\Frob_v$ to be the \emph{geometric} Frobenius at $v$, inducing $x \mapsto x^{1/\Nm(v)}$ on the residue field, although the validity of the above statement is obviously independent of the choice of geometric or arithmetic Frobenius.)
 
 Via the restriction-of-scalars compatibility above, the conjecture is true for the same representations $\Pi$ regarded as automorphic representations of $G = \Res_{F / E} \GL_2$ for any intermediate field $E$, giving admissible homomorphisms 
 \[ r_{\Pi, E, \iota}: \Gamma_E \to {}^LG(\Qlbar).\]
 
 If $\uk$ is not paritious, but is $E$-paritious for some subfield $E$ (recall that this is \emph{always} the case for $E = \QQ$), then the above theorem says nothing. However, as we have seen above, the restriction of $\Pi$ to the group $G^*$ is $L$-algebraic for a suitable choice of $\ut$, and hence the Buzzard--Gee conjecture predicts Galois representations into ${}^L G^*$. The goal of this section will be to construct these ``extra'' Galois representations.
 
\subsection{Representations over CM fields}

 \begin{theorem}[Blasius--Rogawski]
  \label{thm:BR} 
  Let $\Pi$ be a non-CM irreducible subrepresentation of $S_{\uk, \ut}$, where $k_\sigma \ge 2$ for all $\sigma$. Let $K/\QQ$ be an
  imaginary quadratic extension and set $M=FK$. Then there exists a Hecke
  character $\chi$ of $M$, and a continuous Galois representation
  \[ r_{\Pi,\chi, \iota}: \Gamma_M \to \GL_2(\Qlbar), \] 
  with the following property: let $v \nmid \ell$ be a prime of $F$ which splits in $M / F$ and such that $\Pi$ and $\chi$ are unramified at $v$. Then for each of the two primes $w$ above $v$, the restriction of $r_{\Pi, \chi, \iota}$ to $W_{M_w}$ is conjugate to $\iota(s_v \otimes \chi(w))$.  Furthermore, if $\Pi_E$ is not induced from a character of $\AA_M^\times$, then   $r_{\Pi,\chi,\iota}$ is irreducible.
 \end{theorem}
 
 \begin{proof}
  The existence of $r_{\Pi,\chi,\iota}$ comes from \cite[Theorem 2.6.1]{BR},
  while the irreducibility result is proved in the same way as \cite[Theorem 4.14, Proposition 5.9]{Mok} (using the fact that $\Pi$ is assumed to be non-CM, so its base-change to $M$ is cuspidal).
 \end{proof}

 \begin{corollary}
  The representation $\Pi$ is $L$-arithmetic if and only if it is $L$-algebraic.
 \end{corollary}

 \begin{proof}
  As mentioned in
   Remark \ref{rem:LaritvsLalg}, Shimura's algebraicity results show 
   that $L$-algebraic implies $L$-arithmetic. For the converse, the
   argument given in \cite{buzzardgee} generalizes as follows: by
   Theorem \ref{thm:BR} there are infintely many principal primes $v$ for which
   $s_v$ is non-zero (look at the residual representation at a prime
   $\ell \neq 2$ and primes mapping to the identity have this
   property). If $\Pi$ is $L$-arithmetic, by
   Shimura's theorem the set $\{v^{\ut}\Nm(v)\}$ lies in
   a finite extension, so $\ut \in \frac{1}{2}+\ZZ$.
 \end{proof}

 Before stating the main result, we need an auxiliary Lemma.
 \begin{lemma}
   Let $U,V$ be groups, with $Z(V)$ $2$-divisible, and let $U'\subset U$ be
   an index $2$ subgroup. Let $\psi:U'\to V$ be a morphism satisfying:
   \begin{itemize}
   \item it has big image, i.e.
     $\{v \in V \; : \; v \psi(u) v^{-1}=\psi(u) \, \forall u
     \in U'\}=Z(V)$. 
   \item The homomorphism $\psi^\mu: U' \to V$ defined by $\psi^\mu(u)=\psi(\mu u \mu^{-1})$ is conjugate in $V$ to $\psi$.
   \end{itemize}
Then $\psi$ extends to a morphism $U \to V$.
\label{lemma:extension}
 \end{lemma}
 \begin{proof}
   Let $\mu$ be an element of $U - U'$. The second condition means that there exists $v \in V$ such that 
\[
v\psi(u)v^{-1} = \psi(\mu u \mu^{-1})\quad \forall u
     \in U'.
\]
The first condition implies that if such an extension exists, then
$\psi(\mu) = v z$, for some $z \in Z(V)$. The equality
$\psi(\mu^2 u \mu^{-2}) = v^2 \psi(u)v^{-2}$ together with the second
condition implies that $\psi(\mu^2) = v^2 z$ for some $z \in
Z(V)$. Since $Z(V)$ is $2$-divisible, let $\tilde{z} \in Z(V)$ be a
square root of $z$, and define $\psi(\mu)=v \tilde{z}$.
 \end{proof}
 \begin{theorem}
\label{thm:representation}
Let $\Pi$ be a non-CM-type irreducible subrepresentation of $S_{\uk,\ut}$, and
$E\subset F$ such that the restricted representation $\Pi^*$ is $L$-algebraic. Let $\iota : \CC \to \overline{\QQ_\ell}$ an
isomorphism. Then there is a Galois representation 
\[
r_{\Pi, \iota}^*:\Gamma_E \to {}^L G^*(\Qlbar),
\]
whose local factors at unramified places $v$ are the $\iota(r_v^*)$.
\end{theorem}

\begin{proof}
 As in Theorem \ref{thm:BR}, we choose an imaginary quadratic field $K$, and a character $\chi$ of $\AA_M^\times$ (where $M = FK$), such that there is a Galois representation
 \[
  r_{\Pi,\chi,\iota}:\Gamma_M \to \GL_2(\Qlbar)
 \]
 whose Satake parameters at the split primes are determined by $\Pi$ and $\chi$. Let $L = KE$. By Proposition \ref{prop:induceLhom} we can extend $r_{\Pi,\chi,\iota}$
 to an admissible homomorphism
 \[ \tilde r_{\Pi,\chi,\iota}: \Gamma_L \to {}^L G(\Qlbar). \]
 Let us write $r^*_{\Pi,\chi,\iota}$ for the projection of $\tilde r_{\Pi, \chi, \iota}$ into the quotient ${}^L G^*(\Qlbar)$.
 
 Since $\Pi$ is $E$-paritious, the Hecke character $\chi|_{\GL_1(\AA_L)}$ is algebraic. Hence it has a Galois
representation $r_{\chi, \iota}:\Gamma_E \to \GL_1(\Qlbar)$ attached to it. We identify $\GL_1(\Qlbar)$ with the centre of $\hat G^*(\Qlbar)$, and we consider the ``tensor product'' representation 
 \[
  r^*_{\Pi,K,\iota} \coloneqq r^*_{\Pi,\chi,\iota} \otimes r_{\chi^{-1}, \iota}:
  \Gamma_{EK} \to {}^L G(\Qlbar).
 \]
 where by ``tensor product'' we mean the component-wise product in
 $\hat{G}$, which goes to the quotient (as it lies in the center). 
 
 Let us check that this morphism $r^*_{\Pi,K,\iota}$ is independent of the choice of the character $\chi$. If we multiply $\chi$ by an algebraic character $\psi$ of $\AA_M^\times$, then $\psi$ has an associated Galois representation $\Gamma_M \to \GL_1(\Qlbar)$, and we may induce this to a homomorphism $\Gamma_L \to (\GL_1)^{[M : L]} \rtimes \Gal(M / L)$. If we compose this homomorphism with the product map $(\GL_1)^{[M : L]} \to \GL_1$, then the action of $\Gal(M / L)$ becomes trivial, and one checks easily that the result is exactly the Galois representation $\Gamma_L \to \GL_1(\Qlbar)$ associated to $\psi|_{\AA_L^\times}$. Hence the twists cancel out, showing that the representation $r^*_{\Pi,K,\iota}$ is independent of the choice.
 
 Because of the irreducibility of $r_{\Pi, \chi, \iota}$, the centraliser of the image of $r^*_{\Pi,K,\iota}$ is the centre of ${}^L G^*(\Qlbar)$, which is just $\Qlbar^*$ and is thus certainly 2-divisible. So we are in a position to apply the preceding lemma.
 
 Let $\tau$ denote a lift to $\Gamma_E$ of the complex conjugation automorphism of $K / \QQ$. Since $F$ is linearly disjoint from $K$ (and $K$ is Galois), we can and do assume that $\tau$ acts trivially on the dual group $\hat G$. Let $(r^*_{\Pi,\chi, \iota})^\tau$ denote the morphism given by
$(r^*_{\Pi,K,\iota})^\tau(\sigma)=r^*_{\Pi,K,\iota}(\tau \sigma \tau^{-1})$. We claim that $(r^*_{\Pi,K,\iota})^\tau$ is conjugate to $r^*_{\Pi,K,\iota}$. 

Tracing through the definitions, we find that $(r^*_{\Pi,K,\iota})^\tau$ is obtained by induction and twisting from the homomorphism $(r_{\Pi, \chi, \iota})^\tau: \Gamma_M \to \GL_2(\Qlbar)$. Since the representations $(r_{\Pi, \chi, \iota})^\tau$ and $r_{\Pi, \tau(\chi), \iota}$ are both irreducible and their traces agree on the Frobenii at split primes, they are conjugate by an element of $\GL_2(\Qlbar)$. Since the construction of $r_{\Pi, K, \iota}^*$ is independent of the choice of $\tau$, as we have seen, this gives the required conjugacy between $r_{\Pi, K, \iota}^*$ and $(r_{\Pi, K, \iota}^*)^\tau$. Hence $r_{\Pi, K, \iota}^*$ extends to a representation of $\Gamma_E$, uniquely determined up to twisting by the quadratic character associated to $K / \QQ$. 

By construction, $r_{\Pi, K, \iota}^*$ has the desired Satake parameters at all but finitely many primes split in $L / E$. It only remains to prove that the quadratic twists may be chosen in a uniform way, so that the morphisms obtained by extending $r^*_{\Pi,K\iota}$ for different choices of $K$ coincide; this will imply that the resulting representation has the required Satake parameters at every prime (since for any given prime $q$, we may choose $K$ such that $q$ is split in $K$). This will be carried out in the next proposition.
\end{proof}

\begin{proposition}
  Let $K_i$ be an infinite list of imaginary quadratic fields, whose
  ramification set is pairwise disjoint and disjoint from the
  ramification set of $F$, and for each $K_i$ let
  $r^*_{\Pi,K_i,\iota}:\Gamma_{EK_i} \to {}^LG^*(\QQ_\ell)$ be the morphism constructed
  in the previous proof. Then there exists a morphism
  $r_{\Pi,\iota}^* : \Gamma_E \to {}^L{G^*}(\overline{\QQ_\ell})$ whose restriction to $\Gamma_{EK_i}$ is isomorphic to $r^*_{\Pi,K_i,\iota}$ for every $i$. 
\end{proposition}

 \begin{proof} 
  The result resembles that of \cite[Proposition 4.3.1]{BR} and so does its proof. As pointed already each $r^*_{\Pi,K_i,\iota}$ can be extended, non-uniquely, to $\Gamma_E$; let $\varrho_{\Pi, K_i, \iota}$ be such an extension. Note that $\varrho_{\Pi,K_1,\iota}|_{\Gamma_{FK_1K_2}} \simeq  \varrho_{\Pi,K_2,\iota}|_{\Gamma_{FK_1K_2}}$ (using irreducibility, and comparing traces of Frobenii at split primes). Our ramification conditions imply that there are characters $\alpha_{1,2}:\Gal(EK_1/E)\to \CC^\times$ and $\beta_{2,1}:\Gal(EK_2/E)\to \CC^\times$ such that
  \[ 
   \varrho_{\Pi,K_1,\iota} \otimes \alpha_{1,2}\simeq \varrho_{\Pi,K_2,\iota} \otimes \beta_{2,1}.
  \]
  Fix one imaginary quadratic field $K_1$ and let $K_n$ vary. The restriction of $\alpha_{1,2}$ (as a character of $\Gal(\QQbar/EK_1)$) to $\Gal(EK_1K_n/K_n)$ equals that of $\alpha_{1,n}$. Then the representation $\varrho_{\Pi,K_1} \otimes \alpha_{1,2}$ satisfies that its restriction to any $\Gamma_{K_n}$ is isomorphic to $\varrho_{\Pi,K_n,\iota}$, so we define 
  \[
   r^*_{\Pi,\iota} = \varrho_{\Pi,K_1, \iota} \otimes \alpha_{1,2}.\qedhere
  \]
 \end{proof}

 This completes the proof of the Buzzard--Gee conjecture for representations of $G^*$ arising from $E$-paritious Hilbert modular forms.

 \subsection{Realising the Asai representation geometrically}
 
  Composing the representation $r^*_{\Pi, \iota}$ constructed in the preceding subsection with the Asai representation ${}^L G^*(\Qlbar) \to \GL_{2^d}(\Qlbar)$, we obtain a $2^d$-dimensional $\ell$-adic representation of $\Gamma_E$, the \emph{Asai Galois representation} associated to $\Pi$. 
  
  In the special case $E = \QQ$, this representation can be realised geometrically. Attached to the group $G^*$ is a compatible family of Shimura varieties (of varying levels), which are $d$-dimensional algebraic varieties defined over $\QQ$. The main result of \cite{brylinskilabesse84} shows that if the level is taken small enough, the Asai Galois representation of $\Pi$ is realised (up to semisimplification\footnote{If $E = \QQ$ then the semisimplification can be dispensed with, since it has been shown by Nekovar \cite{nekovar} that the $\ell$-adic cohomology is semi-simple.}) as a direct summand of the middle-degree $\ell$-adic intersection cohomology of this Shimura variety (with coefficients in some locally-constant sheaf determined by the weight $\uk, \ut$).  Hence the content of Theorem \ref{thm:representation} is to show that this representation factors naturally through the group ${}^L G^*$.
  
  If $\QQ \subsetneq E \subsetneq F$ then standard conjectures predict that the Asai Galois representation should still be realisable geometrically, via Shimura varieties attached to quaternion algebras. Let us suppose that at least one of the following conditions holds:
  \begin{enumerate}[(i)]
   \item The degree $d = [F : E]$ is even;
   \item The degree $[E : \QQ]$ is odd; 
   \item There is a finite place $v$ of $F$ at which the local factor $\Pi_v$ is in the discrete series.
  \end{enumerate}
   
  We then choose an infinite place $\tau$ of $E$, and a quaternion algebra $B$ over $F$ such that $B \otimes_{F, \sigma} \RR$ is split for $\sigma \mid \tau$ and ramified for all other $\sigma \in \Sigma_F$. If either (i) or (ii) holds there is a unique such $B$ which is unramified at every finite place; if neither (i) nor (ii) holds, but (iii) does, then we can take $B$ to ramify additionally at $v$. Then $\Pi$ admits a Jacquet--Langlands transfer to $B^\times$, and the restriction of this representation to the group $H^*$ of elements of $B^\times$ whose reduced norm is in $E^\times \subset F^\times$ is $L$-algebraic. 
   
  Attached to $H^*$, there is a Shimura variety $\mathcal{X}$ of dimension $d$, whose reflex field is $E$. It is expected that the Asai Galois representation of $\Pi$ should appear in the middle-degree $\ell$-adic cohomology of $\mathcal{X}$, and a conditional proof of this has been given by Langlands \cite{langlands79} modulo a conjecture describing the action of Frobenius on the special fibre.

\section{Relation to Patrikis' construction}
\label{patrikis}

 In the above construction, we verified the Buzzard--Gee conjecture for the restriction of $\Pi$ to the group $G^* \subseteq \Res_{F/E} \GL_2$. One can also restrict further, all the way to the group $G^0 = \Res_{F/E} \SL_2$. This case has also been treated by Patrikis, who works more generally with essentially self-dual automorphic representations of $\GL_n$ and $\SL_n$ for general $n$ \cite[Corollary 5.10]{patrikis-sign}.
 
 For a Hilbert modular automorphic representation $\Pi$, it follows from the $n = 2$ case of Patrikis' result that there is an admissible homomorphism $\Gamma_F \to \operatorname{PGL}_2(\Qlbar)$, or (equivalently, via the restriction-of-scalars formalism of Corollary \ref{cor:resscalars}) an admissible homomorphism $\Gamma_E \to {}^L G_0$, with the appropriate Satake parameters. This can be seen as a consequence of Theorem \ref{thm:representation} by composing with the quotient map ${}^L G^* \to \dfrac{{}^L G^*}{Z(\hat G^*)} = {}^L G^0$.
 
 \begin{remark}
  Patrikis' work suggests that a generalisation of Theorem \ref{thm:representation} should hold for any mixed-parity, regular, essentially self-dual, cuspidal automorphic representation $\Pi$ of $\GL_n / F$. This could potentially be proved, by essentially the same method as above, if one knew that for sufficiently many CM extensions $M$ of $F$, the representations $\Gamma_M \to \operatorname{GL}_n(\Qlbar)$ associated to $L$-algebraic twists of the base change of $\Pi$ to $M$ were irreducible.
 \end{remark}

\section{The case \texorpdfstring{$[F : E] = 2$}{[F : E] = 2}}

 If $F / E$ is a quadratic extension, then the $L$-group ${}^L G^*$ has a particularly simple description. In this case, $\hat G^*$ is the quotient of $\GL_2 \times \GL_2$ by the subgroup of elements of the form $\left( \stbt z00z, \stbt z00z^{-1}\right)$. 
 
 An explicit model for the Asai representation of $\hat G = \GL_2 \times \GL_2$ is given by the action on $2 \times 2$ matrices, via $(g_1, g_2)( m) = g_1 \cdot m \cdot g_2^t$. This factors through $\hat G^*$, and is a faithful representation of $\hat G^*$. We may extend this to a representation of ${}^L G^*$, factoring through the quotient $\hat G^* \rtimes \Gal(F / E)$, by letting the non-trivial element $\sigma \in \Gal(F / E)$ act as $m \mapsto m^t$.
 
 This representation preserves the quadratic form $q(m) = \det m$ up to scalar multiplication, with the multiplier character given by $(g_1, g_2) \mapsto \det(g_1) \det(g_2)$. Thus we may regard this representation as a homomorphism $\hat G^* \rtimes \Gal(F / E) \to \GO_4$. In fact it is an isomorphism between these groups \cite[\S 1]{Ramakrishnan}. The identity component $\GSO_4$ thus corresponds to $\hat G^*$. We thus obtain the following result:
 
 \begin{theorem}
  Let $F / E$ be a quadratic extension of totally real fields, and $\Pi$ a non-CM Hilbert modular automorphic representation of $\GL_2 / F$ whose restriction to $G^*$ is $L$-algebraic. Then, for every embedding $\iota: \QQbar \into \Qlbar$, there exists a Galois representation
  \[ r^*_{\Pi, \iota}: \Gamma_E \to \GO_4(\Qlbar) \]
  such that for primes $w = w_1 w_2$ of $E$ split in $F$,  $r^*_{\Pi, \iota}(\Frob_w)$ is conjugate to the image of $(s_{w_1}(\Pi), s_{w_2}(\Pi))$ under the map $\GL_2 \times \GL_2 \to \GO_4$.
 \end{theorem}
 
 Let $\nu$ denote the orthogonal multiplier $\GO_4 \to \mathbf{G}_m$. Then $\nu \circ r^*_{\Pi, \iota}$ is the $\ell$-adic Galois character corresponding (via $\iota$) to the algebraic Gr\"ossencharacter $\omega |_{\AA_E^\times}$, where $\omega: F^\times \backslash \AF^\times \to \CC^\times$ is the central character of $\Pi$. (Note that $\omega$ will not generally be algebraic as a Gr\"ossencharacter of $F$, but its restriction to $E$ will be.)
 
 The determinant of the standard 4-dimensional representation of $\GO_4$ agrees with $\nu^2$ on $\GSO_4$, but not on $\GO_4$; the determinant of $r^*_{\Pi, \iota}$ is therefore given by $\omega^2 |_{\AA_E^\times} \cdot \chi_{F / E}$, where $\chi_{F / E}$ is the character associated to our quadratic extension.
 
 \begin{remark}
  For $d > 2$ we do not know of a simple description of the image of ${}^L G^*$ in $\GL_{2^d}$.
 \end{remark}

 \section{Computing Hilbert modular forms and Quaternion groups}

 We now explain how these non-paritious Hilbert modular forms can be computed explicitly. For computational purposes, it is better to work with a definite
 quaternion algebra, rather than with the Hilbert modular variety; so we
 need to explain how to explicitly compute examples of non-paritious
 automorphic forms for definite quaternion algebras over $F$, extending the algorithms explained in \cite{DV} for the paritious case.
 
 \subsection{Groups} 
 
 Let $B$ be a totally definite quaternion algebra over $F$, of discriminant
 $\fd_B$, and let $\cO_B$ be a maximal order in $B$. Then
 $H = \Res_{F / E} B^\times$ is an algebraic group over $E$; it is an
 inner form of $G = \Res_{F/E} \GL_2$, and in particular it has the
 same $L$-group as $G$.
 
 Let $H^*$ be the fibre product of $H$ with $\GL_1$ over
 $\Res_{F/ E} \GL_1$ (with respect to the reduced norm map
 $H \to \Res_{F / E} \GL_1$); this is an inner form of $G^*$. The
 $E$-paritious Hilbert modular forms will give rise to automorphic forms for $H$ which are not algebraic, but become algebraic while restricted to
 $H^*$. These are exactly the automorphic forms we shall compute.

 \subsection{Automorphic forms for $H$ and $H^*$}
 
  The following definition is standard:
   
  \begin{definition}
   Let $U$ be an open compact subgroup of $H(\AEf) = (B \otimes \AFf)^\times$, and $W$ a finite-dimensional $\CC$-linear representation of $H(E) = B^\times$. The space of automorphic forms for $H$ of weight $W$ and level $U$ is the space $M_W(H; U)$ of functions
   \[ f: (B \otimes \AFf)^\times \to W\]
   satisfying $f(\gamma g u) = \gamma \cdot f(g)$ for all $\gamma \in B^\times$ and $u \in U$.
  \end{definition}
  
  As is well known, $B^\times \backslash (B \otimes \AFf)^\times / U$ is finite. If $C_U$ denotes a set of representatives for this set, and for $x \in C_U$ we write $\Gamma_x = B^\times \cap x U x^{-1}$, then the map $f \mapsto (f(x))_{x \in C_U}$ gives an isomorphism
  \begin{equation}
   \label{eq:exp-description1}
   M_W(H; U) \cong \bigoplus_{x \in C_U} W^{\Gamma_x}.
  \end{equation}
  In particular, $M_W(H; U)$ is finite-dimensional.
  
  Similarly, if $U^*$ is an open compact subgroup of $H^*(\AFf)$, and $W$ a representation of $H^*(E)$, we can define a space $M_W(H^*; U^*)$ of automorphic forms for $H^*$ of weight $W$ and level $U^*$.
 
 \subsection{Pullback from $H$ to $H^*$}
   
  If $U$ is an open compact subgroup of $H(\AEf)$, and $U^*$ its intersection with $H^*$, then the inclusion $H^*(\AEf) \into  H(\AEf)$ gives a map
  \begin{equation}
   \label{eq:psi}
   \psi: H^*(E) \backslash H^*(\AEf) /U^*  \to H(E) \backslash H(\AEf)/U.
  \end{equation}
  
  \begin{definition}
   The map $\psi$ induces a pullback map
   $\psi^*:M_W(H; U) \to M_{W}(H^*; U^*)$
   given on $f \in M_W(H; U)$ by
   \[\psi^*(f)(x) \coloneqq f(\psi(x)).\]
  \end{definition}
  
  We shall now analyse this map more closely, under the following hypothesis: \emph{the image of $U$ under the reduced norm map $\nrd: H(\AEf) \to \AFf^\times$ is the maximal compact subgroup $\widehat\cO_F^\times$.} For instance, this is true if $U = \widehat\cO_B^\times$, or if $U$ is one of the subgroups $U_1(\fN)$ or $U_0(\fN)$ to be introduced below. In this case, all three maps
   \begin{align}
    \label{eq:surjns}
    H(\AEf) &\to \AEf^\times, & U & \to \widehat\cO_F^\times, & 
    H(E) &\to F^{\times +}
   \end{align}
   induced by the reduced norm are surjective. We thus obtain a surjection from $H(E) \backslash H(\AEf)/U$ to
  $F^{\times +} \backslash \AFf^\times / \widehat \cO_F^\times$, which is the
  narrow class group $\Cl^+(F)$; and this fits into a
  commutative diagram
 \[\begin{tikzcd}
  H^*(E) \backslash H^*(\AEf) /U^* \arrow{r}{\psi} \arrow{d}{\nrd} & 
  H(E) \backslash H(\AEf)/U \arrow{d}{\nrd} \\
  \Cl^+(E) \arrow{r} & \Cl^+(F)
  \end{tikzcd}
  \]
 where the vertical arrows are natural surjections. 
 
 \begin{lemma}
    The image of $\psi$ consists of those elements of
    $H(E) \backslash H(\AEf) / U$ whose reduced norm lies in the image of
    $\Cl^+(E)$ in $\Cl^+(F)$.
 \end{lemma}

  \begin{proof} 
   It is clear from the commutativity of the diagram that the image of $\psi$ cannot be any larger than this. Conversely, let $x \in H(\AEf)$ be such that the class of $\nrd(x)$ is in the image of $\Cl^+(E)$. Since the maps \eqref{eq:surjns} are surjective, there exist $\gamma \in H(E)$ and $u \in U$ such that $\nrd(\gamma x u) \in \AEf^\times$. That is, $\gamma x u \in H^*(\AEf)$, and $\gamma x u$ lies in the same double coset as $x$. 
  \end{proof}
  
  We now study the fibres of $\psi$. We will need the following definition:
  
  \begin{definition}
   The \emph{capitulation group} is the group
   \[ K_{F / E} \coloneqq \frac{ F^{\times +} \cap  \left[ \widehat\cO_F^\times \cdot\AEf^\times\right] }{E^{\times +}}.\]
  \end{definition}
  
  Clearly, if $a \in F^{\times +}$ represents a class in the capitulation group, then the ideal $a \cO_F$ is the base-extension to $\cO_F$ of an ideal of $\cO_E$, whose narrow ideal class is independent of the representative $a$ and is in the kernel of the natural map $\Cl^+(E) \to \Cl^+(F)$ (the \emph{capitulation kernel}). This gives an exact sequence
  \[ 0 \to \frac{\cO_F^{\times +}}{\cO_E^{\times +}} \to K_{F / E} \to \Cl^+(E) \to \Cl^+(F).\]
  
  \begin{definition}
   We define an action of $K_{F/E}$ on
  $H^*(E) \backslash H^*(\AEf) /U^*$ as follows. Given
  $a \in F^{\times+}$ representing a class in $K_{F/E}$, there exists $\gamma \in H(E)$ such that $\nrd(\gamma) = a$, and $u \in U$ such that $a \nrd(u) \in \AEf^\times \subset \AFf^\times$. Then we define the action by
  \[ a \cdot [x] = [\gamma x u],\]
  which clearly is independent of the choice of $\gamma$ and $u$, and preserves the fibres of $\psi$.
  \end{definition}
 
  \begin{remark}
    If $a \in (\cO_F^\times)^2$ then the action is trivial, since for such $a$ we may choose $\gamma$ to be in $Z(B) \cap U$ and $u = \gamma^{-1}$. Thus the action of $K_{F/E}$ factors through the quotient of $K_{F/E}$ by the image of $(\cO_F^\times)^2$, which is a finite group.
  \end{remark}
  
  For $x \in H(\AEf)$, let $\Gamma_x$ denote the group $B^\times \cap x Ux^{-1}$, as above. Let
  $\cO_x = \{\nrd(\nu) \; : \; \nu \in \Gamma_x \} \subset \cO_F^{\times +}$.
   As $(\cO_F^\times)^2 \subset \cO_x$, the quotient $\cO_F^{\times +}/\cO_x$ is
   finite.
 
  \begin{theorem}
  \label{thm:fibers}
   Let $x \in H^*(\AEf)$. Then $K_{F/E}$ acts transitively on $\psi^{-1}(\psi(x))$, and the stabiliser of $x$ is $\cO_x$; i.e.\ the fiber at $\psi(x)$ is an homogeneous space for
   $K_{F/E}/\cO_x$.
  \end{theorem}
  
  \begin{proof} Let $x,y \in H^*(\AEf)$ be such that
    $\psi([x]) = \psi([y])$.  Then there exists $\gamma \in H(E)$ and
    $u \in U$ such that $\gamma x u = y$, so $\nrd(\gamma) \in K_{F/E}$
    and $[y] = \nrd(\gamma) \cdot [x]$, proving that the action is
    transitive.
  
 Clearly the quotient $K_{F/E}/\cO_F^{\times +}$ permutes different fibers, so the stabilizer is contained in $\cO_F^{\times +}/\cO_E^{\times +}$. Let $f \in \cO_F^{\times +}$, choose $\gamma$ and $u$ depending on $f$ as above, and suppose that
    there exists $\tilde{\gamma} \in H^*(E)$ and $\tilde u \in U^*$ such
   that $\gamma x u = \tilde \gamma x \tilde u$. Taking norms,
    $\nrd \tilde \gamma \in \cO_F^{\times +} \cap
    E^\times=\cO_E^{\times +}$. The equality
  \[
  x(\tilde u u^{-1}) x^{-1} = \tilde \gamma^{-1} \gamma,
 \]
  implies that the element on the right belongs to $\Gamma_x$ and has norm
  equal to $\nrd(\gamma)$, up to $\cO_E^{\times +}$. If there is no
  such element, the orbits cannot be equivalent, while if such an element
  $\xi$ exists, $\tilde \gamma = \gamma \xi^{-1} \in H^*(E)$ and
  $\tilde u = x^{-1} \xi x u \in U^*$ gives the required equivalence.
 \end{proof}
 
 \begin{corollary}
\label{corollay:algorithm}
There exist an algorithm to compute the space $M_W(H^*; U^*)$.
\end{corollary}
\begin{proof}
 The action of $K_{E / F}$ on the above double quotients translates readily into an action on the space $M_W(H^*; U^*)$. For $a \in F^{\times +}$ representing a class in $K_{E/F}$, and $\gamma, u$ as before, and $f \in M_W(H^*; U^*)$, , we define 
  \[ (a \cdot f)(x) = \gamma^{-1} f(\gamma x u). \]
  From Theorem \ref{thm:fibers}, we see that the image of the pullback map $\psi^*$ consists of exactly those forms in $M_W(H^*; U^*)$ which are invariant under the action of $K_{F/E}$. Therefore, provided we have determined the image of $\Cl^+(E)$ inside $\Cl^+(F)$ and the capitulation group $K_{F/E}$, the algorithms described in \cite{DV} can be readily adapted to work with $\psi^*\left(M_W(H; U)\right)$.
\end{proof}

 \subsection{Weights}
  
  We now define the specific modules $W$ in which we are interested.
 
  \begin{definition}
   For $(\uk, \ut) $ a weight, with all $k_\sigma \ge 2$, we define the \emph{weight module} of weight $(\uk, \ut)$ to be the $\CC$-linear representation $W(\uk, \ut)$ of $B^\times$ given by
   \[ W(\uk, \ut) = \bigotimes_{\sigma \in \Sigma_F}\left( \operatorname{Sym}^{k_\sigma-2}(V_{\sigma})\otimes (\sigma \circ \nrd)^{2-k_\sigma-t_\sigma} \right).\]
  \end{definition}
  
  (The appearance of $\nrd^{2-k_\sigma-t_\sigma}$ is needed in order for our parametrisation of the weights to be consistent with automorphic forms for $\GL_2$ via the Jacquet--Langlands correspondence.) Here the action of $B^\times$ on the first factor is given by choosing splittings $B \otimes_{F, \sigma} \CC \cong M_{2\times 2}(\CC)$, for each $\sigma \in \Sigma_F$. This representation is, of course, not algebraic unless the $t_\sigma$ are all in $\ZZ$.
  
  \begin{notation}
   We write $M_{\uk, \ut}(H; U)$ for $M_{W(\uk, \ut)}(H; U)$ and similarly for $H^*$.
  \end{notation}
  
  The restriction map $\psi^*$ is clearly compatible with taking direct limits as $U$ shrinks. So we have a well defined map
  \[\psi^* : M_{\uk, \ut}(H) \to M_{\uk, \ut}(H^*),\] 
  where $M_{\uk, \ut}(H) \coloneqq \varinjlim_{U} M_{\uk, \ut}(H; U)$ and likewise for $H^*$.
  
  We now recall the precise statement of the Jacquet--Langlands correspondence. Let $S_{\uk, \ut}(H) = M_{\uk, \ut}(H)$ if $\uk \ne (2, \dots, 2)$, and if $\uk = (2, \dots, 2)$ let it be the quotient of $M_{\uk, \ut}(H)$ by its unique one-dimensional subrepresentation. 
  
  \begin{theorem}[Jacquet--Langlands]
   There is a bijection between the $H(\AEf)$-subrepresentations of $S_{\uk, \ut}(H)$, and the $\GL_2(\AFf)$-subrepresentations of the space $S_{\uk, \ut}$ of holomorphic Hilbert modular forms whose local factors at the primes dividing $\fd_B$ are discrete series; and this bijection preserves Satake parameters at the unramified primes.
  \end{theorem}
  
  Let $\Pi_{H^*}$ be an automorphic representation of $H^*$ of weight $(\uk, \ut)$ which arises from $\psi^*(S_{\uk, \ut}(H))$. Then $\Pi_{H^*}$ is a constituent of some automorphic representation $\Pi_{H}$ of $H$, which is the Jacquet--Langlands correspondent of an automorphic representation $\Pi_G$ of $G$ arising in $S_{\uk, \ut}$. If $\Pi_{G^*}$ is any $G^*$-constituent of $\Pi_G$, then the Satake parameters of $\Pi_{G^*}$ at unramified primes are the same as those of $\Pi_{H^*}$; and we can compute these using the action of Hecke operators on $M_{\uk, \ut}(H^*)$. This gives an explicit approach to computing with automorphic representations arising from (possibly non-paritious) Hilbert modular forms.
    
  \subsection{Induction and Shapiro's lemma}
  
  We shall also need to consider some more general modules
  incorporating some finite-order character. Let $\fN$ be an ideal of $\cO_F$ coprime to $\fd_B$. For each $\fq \mid \fN$ we fix an isomorphism
  \[ \cO_{B, \fq}^\times =(\cO_B \otimes_{\cO_F} \cO_{F, \fq})^\times \cong \GL_2(\cO_{F, \fq}),\]
  so that we can define the subgroups
  $U_0(\fN) = \{ u \in \widehat{\cO}_{B}^\times : u = \stbt * * 0 * \bmod
  \fN\}$
  and
  $U_1(\fN) =\{ u \in \widehat{\cO}_{B}^\times : u = \stbt * * 0 1 \bmod
  \fN\}$.
  Clearly $U_1(\fN) \trianglelefteq U_0(\fN)$, and the quotient is isomorphic to $(\cO_F / \fN)^\times$.
  
  \begin{definition}
   Let $\varepsilon$ be a character of $(\cO_F / \fN)^\times$. The weight module for $(\fN,\uk,\ut, \varepsilon)$ is the $\CC$-linear representation of $B^\times \cap \prod_{\fq \mid \fN} \widehat{\cO}_{B, \fq}^\times$ given by
   \[
   V(\fN,\uk,\ut, \varepsilon) \coloneqq W(\uk, \ut) \otimes \CC[\mathbf{P}^1(\cO_F/\fN)],
   \]
   where the action on $\CC[ \mathbf{P}^1(\cO_F/\fN)] = \CC[ \widehat{\cO}_{B}^\times / U_0(\fN)]$ is given by induction from the character $\varepsilon: U_0(\fN) /U_1(\fN) \to \CC^\times$.
  \end{definition}
   
  The module $V(\fN,\uk,\ut, \varepsilon)$ is not a representation of $B^\times$, but only of the subgroup consisting of elements that are units locally at the primes dividing $\fN$. However, by weak approximation, an automorphic form for $H$ or $H^*$ (of any level) is uniquely determined by its values on elements of $H(\AEf)$ or $H^*(\AEf)$ that are units at $\fN$. Thus we may make the following definition:
   
  \begin{definition}
   We define the space of quaternionic Hilbert modular forms of weight $(\uk, \ut)$, level $\fN$ and character $\eps$ by 
   \[ M_{\uk, \ut}(\fN, \varepsilon) \coloneqq M_{V(\uk, \ut, \fN, \eps)}(H, \widehat \cO_B^\times).\]
   We define similarly a space $M_{\uk, \ut}^*(\fN, \eps)$ of automorphic forms on $H^*$.
  \end{definition}
  
  From Shapiro's lemma, one sees readily that there is an isomorphism between $M_{\uk, \ut}(\fN, \varepsilon)$ and the subspace of $M_{W(\uk, \ut)}(H; U_1(\fN))$ where the quotient $U_0(\fN) / U_1(\fN)$ acts via the character $\varepsilon$. However, the former interpretation is more convenient for computations, since for $U = \widehat\cO_B^\times$ the double cosets $C_U$ have an interpretation as equivalence classes of right $\cO_B$-ideals in $B$, and there are robust algorithms available for computing with them, as explained in \cite{DV}. 
     
  \begin{lemma}
   \label{lemma:char}
   The group $\cO_F^\times \subseteq \cO_B^\times$ acts via a character on $V(\fN,\uk,\ut, \varepsilon)$, and this character is trivial if and only if $(\uk, \ut)$ is reasonable and $\eps(u) = \prod_{\sigma} \operatorname{sign} \sigma(u)^{k_\sigma}$ for all $u \in \cO_F^\times$.\qed
  \end{lemma}
  
  \begin{remark}
   The conditions of the lemma are equivalent to $\varepsilon$ being the finite part of a Hecke character of conductor $\fN$, whose signs at the infinite places are determined by the $k_\sigma$. 
   
   For $U = \widehat\cO_B^\times$, each of the groups $\Gamma_x$ appearing in \eqref{eq:exp-description1} will contain $\cO_F^\times$ as a finite-index subgroup; so $M_{\uk, \ut}(\fN, \varepsilon)$ is zero unless the conditions of Lemma \ref{lemma:char} are satisfied. If these conditions do hold, then $M_{\uk, \ut}(\fN, \varepsilon)$ can be decomposed into a direct sum of eigenspaces for the action of $Z(H)(\AEf)$, corresponding to the set of Gr\"ossencharacters of $F$ extending $\eps$.
  \end{remark}
  
 \subsection{Hecke operators}
  \label{sect:hecke}
  
  Let $\fm$ be an ideal of $\cO_F$ coprime to $\fN \fd_B$. On the space $M_{\uk, \ut}(\fN, \eps)$, we have the following Hecke operators:
  \begin{itemize}
   \item The operator $\cT(\fm)$, given by the double $U$-coset of elements of $\widehat\cO_B$ whose norms generate the ideal $\fm\widehat\cO_F$;
   \item the operator $\cS(\fm)$, given by the double $U$-coset generated by the element $x \in Z(H)(\AEf)$, for any $x \in \widehat\cO_F$ generating the ideal $\fm\widehat\cO_F$.
  \end{itemize}
  They satisfy the familiar multiplicative relations: if $\fm$ and $\fm'$ are coprime, then $\cT(\fm\fm') = \cT(\fm) \cT(\fm')$, and if $\fp$ is prime, then $\cT(\fp)^2 = \cT(\fp^2) + q \cS(\fp)$, where $q = \Nm(\fp)$. If $\fm$ is narrowly principal, generated by some $x \in F^{\times+}$, then $\cS(\fm) = \Nm(x)^{2-R} \eps(x)$.
  
  For $M^*_{\uk, \ut}(\fN, \eps)$, the action of Hecke operators is more restricted. We obtain Hecke operators $\cT(\fm)$ and $\cS(\fm)$ for any ideal $\fm$ of $\cO_E$ (rather than $\cO_F$) coprime to $\fN \fd$, and these are compatible with the corresponding operators for $H$ via the map $\psi$. More generally, we can descend to $H^*$ those Hecke operators for $H$ corresponding to double cosets with a natural choice of representative lying in $H^*$. For instance, if $\fp$ is a prime of $F$, then the operator $\cS(\fp)^{-1} \cT(\fp^2)$ is well-defined as a Hecke operator for $H^*$, although $\cS(\fp)$ and $\cT(\fp^2)$ themselves are not, since in the spherical Hecke algebra of $\GL_2(F_\fq)$ we have
  \[ \cS(\fp)^{-1} \cT(\fp^2) = [1] + \left[ \stbt{\varpi^{-1}} 0 0 \varpi\right] \]
  for $\varpi$ a uniformizer at $\fq$, and the double-coset representatives on the left are in $\SL_2(F_{\fq})$ and thus a fortiori in $H^*(\AEf)$.
    
  Although we have fewer Hecke operators to consider when working with $H^*$, we have potentially gained an algebraicity property. If $\uk$ is not $F$-paritious, but is $E$-paritious, then we can choose $\ut$ such that $(\uk, \ut)$ is reasonable and $W$ is algebraic as a representation of $H^*$ (although we cannot, of course, make it algebraic as a representation of $H$). In this case, we can find a finite extension $L / \QQ$ to which $V(\fN, \uk, \ut, \eps)$ descends, and hence $M^*_{\uk, \ut}(\fN, \eps)$ is the base-extension to $\CC$ of an $L$-vector space which is preserved by the action of the Hecke operators for $H^*$.
  
  \begin{remark}
   \label{rmk:shimura}
   We can re-introduce some of the ``missing'' Hecke action using a trick due to Shimura (cf.\ \cite[Definition 2.2.4]{LLZ}). Let $\mathscr{H}$ denote the subgroup of $(B \otimes \AFf)^\times$ consisting of the elements whose reduced norms are in $F^{\times +} \cdot \AEf^\times \subset \AFf^\times$. Then the double quotient $H(E) \backslash \mathscr{H} / U^*$ bijects with $H^*(E) \backslash H^*(\AEf) / U^*$, so we can interpret $M_{\uk, \ut}^*(\fN, \eps)$ as a space of functions on $\mathscr{H} / U^*$. Thus we may define a Hecke operator for any double $U^*$-coset in $\mathscr{H}$. In particular, we can use this to make sense of $\cT(\fp)$ as an operator on $M^*_{\uk, \ut}(\fN, \eps)$ for any prime $\fp \nmid \fN\fd_B$ of $F$ whose ideal class lies in the image of $\Cl^+(E)$ in $\Cl^+(F)$; however, this will only be well-defined modulo the action of the capitulation group $K_{E/F}$.
   
   Note that the Hecke operators associated to double cosets in $\mathscr{H}$ make sense even if $(\uk, \ut)$ is not ``reasonable'' in the sense of \S \ref{sect:adelicHMF}, since we only need $\cO_E^\times$ to act trivially, not $\cO_F^\times$. We shall see an application of this in the next section.
  \end{remark}
   
\section{An explicit example of a non-paritious Hilbert eigenform}

 \subsection{Setup}
 
 Let $F = \QQ(\sqrt{2})$, and let $\sigma_1, \sigma_2$ denote the two
 embeddings $F \into \RR$ (mapping $\sqrt{2}$ to $\sqrt{2}$ and
 $-\sqrt{2}$ respectively). Let $B = \left( \frac{-1, -1}{F} \right)$ be the Hamilton quaternions over $F$, so that $B$ is the unique quaternion
 algebra over $F$ unramified at all finite places; and let $\cO_B$ be a maximal order
 in $B$, so that $\widehat\cO_B^\times$ is a maximal compact subgroup
 of $H(\AFf)$. The class number of $\cO_B$ is one.
 
 There is a 6-dimensional $\CC$ representation of the group
 $H = \Res_{F / \QQ} B^\times$ corresponding to $\uk = (4, 3)$ and
 $\ut = (-\tfrac{7}{4}, -\tfrac{5}{4})$, given by
 \[ 
  W = \Sym^2 V_{\sigma_1} \otimes \Sym^1 V_{\sigma_2} \otimes (\sigma_1 \circ \nrd)^{-1/4} \otimes (\sigma_2 \circ \nrd)^{1/4}, 
 \]
 where $V_{\sigma_i}$ is the 2-dimensional representation of $H$
 coming from a splitting of $B \otimes_{F, \sigma_i} \CC$. This
 representation is, of course, not algebraic, but its restriction to $H^*$ is algebraic and can be descended
 to any finite extension $K / F$ over which $B$ splits, such as the
 cyclotomic field $\QQ(\zeta_8)$.
 
 The central character of $W$ is the character of $Z(B^\times) = F^\times$ given by
 \[ z \mapsto \sigma_1(z)^2 \cdot \sigma_2(z) \cdot |\sigma_1(z)^2|^{-1/4} \cdot |\sigma_2(z)^2|^{1/4} = |\Nm_{F / \QQ} z|^{3/2} \operatorname{sign} \sigma_2(z).\]
 
 In order to obtain non-zero Hilbert modular forms, we need to take a non-trivial character. Let $\fN$ be the ideal generated by $5 - 3\sqrt{2}$ (so $\fN$ is one of the two prime ideals above $7$). There is a unique non-trivial quadratic character $\eps: (\cO_F /\fN)^\times \to \pm 1$, and one checks that for $u \in \cO_F$ we have $\eps(u) = \operatorname{sign} \sigma_2(u)$, where $\sigma_2$ is the embedding $F \into \RR$ mapping $\sqrt{2}$ to $-\sqrt{2}$; in particular, the restriction of $\eps$ to $\cO_F^\times$ is the inverse of the central character of $V$, a necessary condition for Hilbert modular forms of weight $V$ and character $\eps$ to exist. 
 
 With this choice we compute that the space $M_{\uk, \ut}(\fN, \eps)$ is 2-dimensional. Since $F$ has narrow class number one, and $\cO_F^{\times +} = (\cO_F^\times)^2$, this is isomorphic (via the pullback map $\psi$) to the space $M^*_{\uk, \ut}(\fN, \eps)$. 
 \subsection{Hecke operators}
  
  If $\mathfrak{m}$ is an ideal of $F$ coprime to $\fn$, then we have two related definitions of a Hecke operator at $\mathfrak{m}$:
  \begin{itemize}
   \item A \emph{normalized Hecke operator} $\mathcal{T}(\mathfrak{m})$, defined as in \S \ref{sect:hecke} above.
   \item A \emph{naive Hecke operator} $T(\varpi)$, depending on a choice of totally-positive generator $\varpi$ of $\mathfrak{m}$. This is given by identifying $W$ as an $H^*$-representation with the representation $W(\uk, \ut') = \Sym^2 V_{\sigma_1} \otimes \Sym^1 V_{\sigma_2}$, where $\ut' = 2 - \uk = (-2, -1)$; and treating $T(\varpi)$ as a double coset in the group $\mathscr{H}$ of Remark \ref{rmk:shimura}. 
  \end{itemize} 
  The normalisation of the ``naive Hecke operator'' is chosen in such a way that its eigenvalue corresponds to the ``naive Hecke eigenvalue'' defined above in the complex-analytic theory. The two operators are related by the formula 
  \begin{equation}
   \label{eq:hecke}
    \mathcal{T}(\mathfrak{m}) = \left(\frac{\sigma_2(\varpi)}{\sigma_1(\varpi)}\right)^{\tfrac{1}{4}} T(\varpi).
   \end{equation}
  In particular, if $\fm$ is the base-extension to $F$ of an ideal of $\ZZ$, and $\varpi$ is the positive integer generating $\fm$, then $\cT(\fm)$ and $T(\varpi)$ agree.
  
  The normalised Hecke operator $\mathcal{T}(\mathfrak{m})$ is canonically defined, but it does not preserve the natural $K$-structure on the space, so the collection of eigenvalues of these operators (for varying $\mathfrak{m}$) do not all lie in a finite extension of $\QQ$. On the other hand, the naive Hecke operator $T(\varpi)$ preserves the $K$-structure, but it will depend on the the choice of generator $\varpi$. 
  
  From equation \eqref{eq:hecke}, it is clear that if $p$ is a prime inert in $F$ and $\mathfrak{m} = (p)$, then $\mathcal{T}(\mathfrak{m}) = T(p)$; whereas if $p = \fp_1 \fp_2$ is a prime split or ramified in $F$, and $\varpi_1, \varpi_2$ are totally positive generators of these ideals such that $\varpi_1 \varpi_2 = p$, then $\mathcal{T}(\fp_1) \mathcal{T}(\fp_2) = T(\varpi_1) T(\varpi_2) = T(p)$. So in either case we \emph{do} have a canonical operator $T(p)$, which is both independent of choices and has eigenvalues defined over a finite extension, which is the Hecke operator of $H^*$ and can be computed with either definition. 

  Similarly we can define a normalized operator $\mathcal{S}(\mathfrak{m})$ for any ideal $\mathfrak{m}$, and a naive operator $S(\varpi)$ for $\varpi \in \cO_F$, via the action of $\stbt \varpi 0 0 \varpi$. Note that if $p$ is a split prime and $\varpi_1 \varpi_2 = p$, the operators $T(\varpi_1^2)S(\varpi_2)$ and $T(\varpi_2^2)S(\varpi_1)$ are well defined and are independent of the choice of generators with either (but consistent) definition. Clearly the action of $S(p)$ is given by $p^3\varepsilon(p)$. 
  
 \subsection{Hecke eigenvalues} 
 
  Our space $M_{\uk,\ut}(\fN,\varepsilon)$ is an irreducible module for the Hecke algebra with coefficients in $F$; it decomposes over the CM field $L = F[b]$, where $b^2 = -3\sqrt{2}-8$. (We note that $L$ is not Galois over $\mathbf{Q}$.) 
  
  In Table \ref{table:lev7-sqrt2}, we display the Hecke eigenvalues for all primes of $F$ of norm up to 200. For an inert prime $p$, we list the eigenvalue $t(p)$ of the Hecke operator $T(p) = \cT(p)$. For a split prime, we choose arbitrary totally-positive generators $\varpi_1$ and $\varpi_2$ of the two primes above $p$ such that $\varpi_1 \varpi_2 = p$, and we list the eigenvalues $t(\varpi_i)$ of the naive Hecke operators $T(\varpi_1)$ and $T(\varpi_2)$.
  
  The eigenvalues displayed show many of the interesting features we expect for such an eigensystem. For example, we see that the eigenvalue $t(\varpi)$ lies in $F$ when $\varepsilon(\varpi) = 1$, and in $b\cdot F$ when $\varepsilon(\varpi) = -1$. In particular, when $p$ is totally split in $\mathbf{Q}(\sqrt{2}, \sqrt{-7})$, such as $p = 23$, then we see that $t(\varpi_1)$ and $t(\varpi_2)$ are both in $F$.

  The smallest rational prime which is inert in $F$ is $p = 3$. In that case, we have $\varepsilon(3) = -1$, and $t(3) =  (7\sqrt{2} - 4)b$. 
 
  The smallest rational prime which splits in $F$ is $p = 17$: we have $17 = \varpi_1 \varpi_2$ where $\varpi_1 = 2\sqrt{2} + 5$. Note that $\eps(\varpi_1) = -1$, but $\eps(\varpi_2) = +1$, so $t(\varpi_2)$ is in $F$ but $t(\varpi_1)$ is not, and nor is the product $t(p) = t(\varpi_1) t(\varpi_2) =  (150\sqrt{2} + 264)b$ is not in $F$. 
  
  If $\fp_1 = (\varpi_1)$ then equation \eqref{eq:hecke} tells us that the normalised Hecke operator $\cT(\fp)$-eigenvalue acts as $(3\sqrt{2} + 12)b \cdot \left( \frac{ 5 - 2\sqrt{2}}{5 + 2\sqrt{2}} \right)^{1/4}$. 
  Any other totally positive generator of $\fp$ is of the form $\varpi'= \varpi u^{2k}$, where $u = 1 + \sqrt{2}$ is the fundamental unit. For such a generator,
  we see that $T(\varpi') = (3\sqrt{2} + 12) u^k b$, and one readily verifies that
  \[ 
   (3\sqrt{2} + 12)u^k \cdot  \left( \frac{ (5 - 2\sqrt{2})u^{-2k}}{(5 + 2\sqrt{2})u^{2k}} \right)^{\frac{1}{4}} 
   =(3\sqrt{2} + 12) \cdot \left( \frac{ 5 - 2\sqrt{2}}{5 + 2\sqrt{2}} \right)^{\frac{1}{4}}. 
  \]
  So, indeed, the eigenvalue for the normalised Hecke operator $\mathcal{T}(\fp)$ is independent of the choice of totally positive generator of $\fp$.
 
 \begin{table}
\caption{Naive Hecke eigenvalues at level $(5-3\sqrt{2})$ and weight $(4,3)$ over $\QQ(\sqrt{2})$, for primes of norm $< 200$. Here $w = \sqrt{2}$ and $b^2 = -3\sqrt{2} - 8$.}
\label{table:lev7-sqrt2}\small
\begin{tabular}{ >{$}c<{$}   >{$}r<{$}   >{$}r<{$}  >{$}r<{$} }
\toprule
\Nm(p) & \varpi_1 & t(\varpi_1) &  t(\varpi_2)\\\midrule
9 & 3 & (7w - 4)b \\ 
17 & 2w + 5 & (3w + 12)b &  -8w - 18 \\ 
23 & w + 5 & -22w + 14 &  26w + 36 \\ 
25 & 5 & (-16w + 18)b \\ 
31 & 3w + 7 & (13w - 18)b  & -30w + 34 \\ 
41 & 2w + 7 & -16w - 106  & (-32w + 26)b \\ 
47 & w + 7 & -76w + 46 & (7w - 70)b \\ 
71 & 5w + 11 & (-74w - 6)b  & (3w - 32)b \\ 
73 & 2w + 9 & (-27w + 18)b  & 168w + 14 \\ 
79 & w + 9 & (-46w + 60)b  & (7w + 40)b \\ 
89 & 4w + 11 & (65w + 64)b  & -206w + 30 \\ 
97 & 6w + 13 & 272w + 38  & (83w - 32)b \\ 
103 & 3w + 11 & 78w + 228  & (-8w + 122)b \\ 
113 & 2w + 11 & (46w - 56)b  & (-18w + 8)b \\ 
121 & 11 & 170w + 366 \\ 
127 & 9w + 17 & -50w + 46  & -272w + 372 \\ 
137 & 14w + 23 & -10 &  -74w + 114 \\ 
151 & 3w + 13 & -282w - 168 & 172w - 318 \\ 
167 & w + 13 & (172w - 166)b & -398w - 24 \\ 
169 & 13 & (-84w + 62)b \\ 
191 & 7w + 17 & (11w + 12)b &  (-114w + 184)b \\ 
193 & 4w + 15 & (129w + 162)b  & (185w - 486)b \\ 
199 & 11w + 21 & -250w - 188  & (-288w + 430)b \\ %
\bottomrule
\end{tabular}
\end{table}

 \subsection{Satake parameters}
 
 Let $\Pi = \Pi_0 \otimes \|\nrd\|^{-1/2}$, where $\Pi_0$ is the automorphic representation of $H$ arising from the system of eigenvalues described above (and tabulated in Table \ref{table:lev7-sqrt2}). The shift by $\|\nrd\|^{-1/2}$ is included in order to give a slightly more pleasant normalisation of the Satake parameters.
  
 If $s_{\fp}$ denotes the Satake parameter of $\Pi$ at a finite prime $\fp$, then $s_\fp$ is the conjugacy class of matrices with characteristic polynomial
 \[ \mathcal{H}_{\fp}(X) = X^2 - \tau(\fp) X + \Nm(p)^{5/2} \eps(\fp), \]
 where $\tau(\fp)$ denotes the $\mathcal{T}(\fp)$-eigenvalue. On the other hand, we may consider the ``naive Satake parameter''
 \[ s_{\varpi} = \left(\frac {\sigma_1(\varpi)}{\sigma_2(\varpi)}\right)^{1/4} s_{\fp}, \]
 where $\varpi$ is a choice of totally-positive generator of $\fp$. Then the characteristic polynomial of $s_{\varpi}$ is the polynomial
 \[ H_\varpi(X) = X^2 - t(\varpi) X + \sigma_1(\varpi)^{3} \sigma_2(\varpi)^2 \eps(\fp) \]
 where as above $t(\varpi)$ is the eigenvalue of $T(\varpi)$; and these polynomials all have coefficients in the finite extension $L = F[b]$.

If $p = \fp_1 \fp_2$ is a rational prime split in $F$, and $\varpi_1, \varpi_2$ are positive generators of the $\fp_i$ chosen so that $\varpi_1 \varpi_2 = p$, then the images of the pairs $(s_{\fp_1}, s_{\fp_2})$ and $(s_{\varpi_1}, s_{\varpi_2})$ in the quotient
\[ \frac{(\GL_2(\CC) \times \GL_2(\CC)) \rtimes \Gal(F / \QQ)}{\{ (z,
    z^{-1}): z \in \CC^\times\}} \cong \GO_4(\CC) \] are
the same. The common image of these elements gives the conjugacy class of $r^*_{\Pi, \iota}(\Frob_p)$. Using this description one can easily compute the characteristic polynomial of $r^*_{\Pi, \iota}(\Frob_p)$ in the standard representation of $\GO_4$: if $p$ is split, it is given by\footnote{Note that for the $X^2$ coefficient we need to compute the Hecke operators $T(p)^2$ and $T(p^2)$; these can be calculated directly as double cosets, but it is quicker computationally to express these operators as polynomials in $T(\varpi_1)$ and $T(\varpi_2)$, since evaluating these non-normalised operators involves summing over fewer double coset representatives.}
\[
H_p(X)=X^4-t(p)X^3+\left(t(p)^2-t(p^2) - p^5 \varepsilon(p)\right)X^2-p^5t(p)\varepsilon(p)X+p^{10}\eps(p)^2.
\]
Similarly, if $p$ is inert in $F$ it is given by
\begin{align*}
H_p(X)
&= X^4 - t(p) X^3 + p^5 t(p) \eps(p) X - p^{10} \eps(p)^{2}.
\end{align*}
The coefficients of these characteristic polynomials for the three smallest primes of each type are given in Table \ref{table:lev7-sqrt2v3}. 

\begin{table}
\caption{Characteristic polynomials of $\Frob_p$ in the standard representation of $\GO_4$ (notations as in Table \ref{table:lev7-sqrt2}).}
\label{table:lev7-sqrt2v3}
\small
\begin{tabular}{ >{$}c<{$}  >{$}l<{$} }
\toprule
p & \multicolumn{1}{c}{$H_p(X)$}\\
\midrule
3 & X^4 + (-7w + 4)bX^3 + (-1701w + 972)bX -3^{10} \\[0.5em]
5 & X^4 + (16w - 18)bX^3 + (50000w - 56250)bX -5^{10} \\ [0.5em]
11 & X^4 + (-170w - 366)X^3 + (27378670w + 58944666)X -11^{10} \\
\midrule
\multirow{2}{*}{$17$} & X^4 + (150w + 264)bX^3 + (-1213222w + 584358)X^2\\
&\hfill  + (-212978550w - 374842248)bX + 17^{10}\\[0.5em]
\multirow{2}{*}{$23$} & X^4 + (428w + 640)X^3 + (4107156w - 157642)X^2\\
&\hfill + (2754754804w +  4119259520)X + 23^{10}\\[0.5em]
\multirow{2}{*}{31} & X^4 + (-982w + 1392)bX^3 + (24199902w + 22262526)X^2\\
&\hfill + (28113826282w - 39851778192)bX + 31^{10}\\
\bottomrule
\end{tabular}
\end{table}

\renewcommand{\MR}[1]{%
 MR \href{http://www.ams.org/mathscinet-getitem?mr=#1}{#1}.
}
\newcommand{\articlehref}[2]{\href{#1}{#2}}


\begin{thebibliography}{Ram02}

\bibitem[Asa77]{asai77}
Tetsuya Asai, \articlehref{http://dx.doi.org/10.1007/BF01391220}{\emph{On
  certain {D}irichlet series associated with {H}ilbert modular forms and
  {R}ankin's method}}, Math. Ann. \textbf{226} (1977), no.~1, 81--94.
  \MR{0429751}

\bibitem[BR93]{BR}
Don Blasius and Jonathan~D. Rogawski,
  \articlehref{http://dx.doi.org/10.1007/BF01232663}{\emph{Motives for
  {H}ilbert modular forms}}, Invent. Math. \textbf{114} (1993), no.~1, 55--87.
  \MR{1235020}

\bibitem[BL84]{brylinskilabesse84}
J.-L. Brylinski and J.-P. Labesse,
  \articlehref{http://www.numdam.org/item?id=ASENS_1984_4_17_3_361_0}{\emph{Cohomologie
  d'intersection et fonctions {$L$} de certaines vari\'et\'es de {S}himura}},
  Ann. Sci. \'Ecole Norm. Sup. (4) \textbf{17} (1984), no.~3, 361--412.
  \MR{777375}

\bibitem[BG14]{buzzardgee}
Kevin Buzzard and Toby Gee,
  \articlehref{http://dx.doi.org/10.1017/CBO9781107446335.006}{\emph{The
  conjectural connections between automorphic representations and {G}alois
  representations}}, Automorphic forms and {G}alois representations. {V}ol. 1,
  London Math. Soc. Lecture Note Ser., vol. 414, Cambridge Univ. Press,
  Cambridge, 2014, pp.~135--187. \MR{3444225}

\bibitem[DV13]{DV}
Lassina Demb{\'e}l{\'e} and John Voight,
  \articlehref{http://dx.doi.org/10.1007/978-3-0348-0618-3_4}{\emph{Explicit
  methods for {H}ilbert modular forms}}, Elliptic curves, {H}ilbert modular
  forms and {G}alois deformations, Adv. Courses Math. CRM Barcelona,
  Birkh\"auser/Springer, Basel, 2013, pp.~135--198. \MR{3184337}

\bibitem[Lan79]{langlands79}
R.~P. Langlands,
  \articlehref{http://dx.doi.org/10.4153/CJM-1979-102-1}{\emph{On the zeta
  functions of some simple {S}himura varieties}}, Canad. J. Math. \textbf{31}
  (1979), no.~6, 1121--1216. \MR{553157}

\bibitem[LLZ18]{LLZ}
Antonio Lei, David Loeffler, and Sarah~Livia Zerbes,
  \articlehref{https://doi.org/10.1017/fms.2018.23}{\emph{Euler systems for
  {H}ilbert modular surfaces}}, Forum Math. Sigma \textbf{6} (2018), e23.

\bibitem[Mok14]{Mok}
Chung~Pang Mok,
  \articlehref{http://dx.doi.org/10.1112/S0010437X13007665}{\emph{Galois
  representations attached to automorphic forms on {$\operatorname{GL}_2$} over
  {CM} fields}}, Compos. Math. \textbf{150} (2014), no.~4, 523--567.
  \MR{3200667}

\bibitem[Nek18]{nekovar}
Jan Nekov\'{a}\v{r}, \articlehref{https://webusers.imj-prg.fr/~jan.nekovar/pu/semi.pdf}{\emph{Eichler--{S}himura relations and semi-simplicity of
  {\'e}tale cohomology of quaternionic {S}himura varieties}}, Ann. Sci. E.N.S. \textbf{51} (2018), 1179--1252.

\bibitem[Pat15]{patrikis-sign}
Stefan Patrikis,
  \articlehref{http://dx.doi.org/10.1007/s00208-014-1111-x}{\emph{On the sign
  of regular algebraic polarizable automorphic representations}}, Math. Ann.
  \textbf{362} (2015), no.~1-2, 147--171. \MR{3343873}

\bibitem[Ram02]{Ramakrishnan}
Dinakar Ramakrishnan,
  \articlehref{http://dx.doi.org/10.1155/S1073792802000016}{\emph{Modularity of
  solvable {A}rtin representations of {$\GO(4)$}-type}}, Int. Math. Res. Not.
  \textbf{2002} (2002), no.~1, 1--54. \MR{1874921}

\bibitem[Shi78]{shimura78}
Goro Shimura,
  \articlehref{http://projecteuclid.org/euclid.dmj/1077312955}{\emph{The
  special values of the zeta functions associated with {H}ilbert modular
  forms}}, Duke Math. J. \textbf{45} (1978), no.~3, 637--679. \MR{507462}

\bibitem[Yos94]{yoshida94}
Hiroyuki Yoshida,
  \articlehref{http://dx.doi.org/10.1215/S0012-7094-94-07505-4}{\emph{On
  the zeta functions of {S}himura varieties and periods of {H}ilbert modular
  forms}}, Duke Math. J. \textbf{75} (1994), no.~1, 121--191. \MR{1284818}

\end{thebibliography}
\end{document}